\documentclass[letterpaper,10pt,reqno]{amsart}
\usepackage[usenames,dvipsnames]{xcolor}

\RequirePackage[OT1]{fontenc}
\usepackage{underscore}
\usepackage{mathrsfs}
\usepackage[colorlinks,citecolor=blue,urlcolor=blue,linkcolor=blue,linktocpage=true, pdffitwindow=true, pdfpagetransition=scroll]{hyperref}
\usepackage[foot]{amsaddr}
\usepackage{amssymb,amsthm,amsfonts,amsbsy,latexsym,dsfont}
\usepackage[textsize=small]{todonotes}       
\usepackage{graphicx}
\usepackage[english]{babel}
\usepackage{subfigure}
\usepackage{appendix}

\usepackage{enumerate}
\newenvironment{enumeratei}{\begin{enumerate}[\upshape (i)]}{\end{enumerate}}

\newenvironment{enumeraten}{\begin{enumerate}[\upshape 1.]}{\end{enumerate}}

\usepackage{color}              
\usepackage{graphicx}
\usepackage[]{amsmath}
\usepackage{amsthm}
\usepackage{amssymb}
\usepackage{bbm}
\usepackage{hyperref}
\usepackage{comment}
\usepackage{microtype}

\definecolor{MyDarkblue}{rgb}{0,0.08,0.50}
\definecolor{Brickred}{rgb}{0.65,0.08,0}

\hypersetup{
colorlinks=true,       
    linkcolor=MyDarkblue,          
    citecolor=Brickred,        
    filecolor=red,      
    urlcolor=cyan           
}

\newtheorem{theorem}{Theorem}[section]
\newtheorem{lemma}[theorem]{Lemma}
\newtheorem{proposition}[theorem]{Proposition}
\newtheorem{corollary}[theorem]{Corollary}

\newtheorem{definition}[theorem]{Definition}

\newtheorem{remark}[theorem]{Remark}
\newtheorem{claim}[theorem]{Claim}

\renewcommand{\P}{\mathbb{P}}

\newcommand{\Pv}{\mathbb{P}}

\newcommand{\E}{\mathbb{E}}
\newcommand{\Ev}{\mathbb{E}}

\newcommand{\bpi}{\mbox{\boldmath$\pi$}}

\newcommand{\cA}{\mathcal{A}}\newcommand{\cC}{\mathcal{C}}
\newcommand{\cE}{\mathcal{E}}\newcommand{\cF}{\mathcal{F}}

\newcommand{\cN}{\mathcal{N}}
\newcommand{\cP}{\mathcal{P}}\newcommand{\cR}{\mathcal{R}}
\newcommand{\cU}{\mathcal{U}}
\newcommand{\cW}{\mathcal{W}}

\newcommand{\bZ}{\mathbf{Z}}

\newcommand{\Var}{{\rm Var}}

\newcommand{\e}{{\mathrm e}}

\numberwithin{equation}{section}

\newcommand{\R}{\mathbb{R}}
\newcommand{\N}{\mathbb{N}}
\newcommand{\Z}{\mathbb{Z}}

\newcommand{\Poi}{\mathrm{Poi}}

\renewcommand{\emptyset}{\varnothing}

\newcommand{\CP}{\mathcal {P}}

\newcommand{\CH}{\mathcal {H}}

\newcommand*{\la}{\lambda}

\newcommand*{\ve}{\varepsilon}

\newcommand*{\Cov}{{\text{\bf Cov}}}

\newcommand*{\be}{\begin{equation}}
\newcommand*{\ee}{\end{equation}}
\newcommand*{\ba}{\begin{aligned}}
\newcommand*{\ea}{\end{aligned}}
\newcommand*{\barr}{\begin{array}{c}}
\newcommand*{\earr}{\end{array}}

\def \toindis  {\buildrel {d}\over{\longrightarrow}}
\def \toas     {\buildrel {a.s.}\over{\longrightarrow}}

\newcommand*{\ind}{\mathbbm{1}}
\def\namedlabel#1#2{\begingroup
    #2%
    \def\@currentlabel{#2}%
    \phantomsection\label{#1}\endgroup
}

\newcommand{\Exp}{\mathrm{Exp}}
\newcommand{\NWn}{\mathrm{NW}_n}
\newcommand{\SWT}{\mathrm{SWT}}

\newcommand{\Poisson}{\mathrm{Poisson}}

\newcommand{\Bin}{\mathrm{Bin}}
\newcommand{\Binomial}{\mathrm{Binomial}}

\renewcommand{\Cov}{\mathrm{Cov}}

\newcommand{\WB}{W^{\scriptscriptstyle{(B)}}}
\newcommand{\WR}{W^{\scriptscriptstyle{(R)}}}
\newcommand*{\Wun}{W_U^{\scriptscriptstyle{(n)}}}
\newcommand*{\Wvn}{W_V^{\scriptscriptstyle{(n)}}}
\newcommand{\cWn}{\cW^{\scriptscriptstyle{(n)}}}

\newcommand{\rN}{\mathrm{N}}
\newcommand{\rA}{\mathrm{A}}

\newcommand{\rM}{\mathrm{M}}
\newcommand{\rI}{\mathrm{I}}

\newcommand{\rH}{\mathrm{H}}
\newcommand{\rQ}{\mathrm{Q}}
\newcommand{\rC}{\mathrm{C}}
\newcommand{\rAL}{\mathrm{AL}}
\newcommand{\rd}{\mathrm{d}}
\newcommand{\actset}[1]{\left\{\mathcal{A}(#1)\right\}}
\newcommand{\ith}{^\text{th}}

\newcommand{\bba}{\mathbf{a}}
\newcommand{\bu}{\mathbf{u}}
\newcommand{\bv}{\mathbf{v}}

\def \eqas   {\buildrel \text{a.s.}\over{=}}

\def \eqindis   {\buildrel \text{d}\over{=}}

\def \toninf {\buildrel {n\to\infty}\over{\longrightarrow}}

\begin{document}
	\title[FPP on the Newman Watts model]{First passage percolation on the Newman-Watts\\ small world model}

	\date{\today}

	\author[Komj\'athy]{J\'ulia Komj\'athy{$^{1,*}$}}
	\address{$1$ Department of Mathematics and
	    Computer Science, Eindhoven University of Technology}
   \author[Vadon]{Vikt\'oria Vadon{$^2$}}
    \address{$2$ Department of Stochastics, Budapest University of Technology and Economics}
	\email{j.komjathy@tue.nl, vadonvik@math.bme.hu}
\thanks{$^*$ This work is part of the research programme Veni (project number 639.031.447), which is (partly) financed by the Netherlands Organisation for Scientific Research (NWO)}

\begin{abstract}
The Newman-Watts model is given by taking a cycle graph of $n$ vertices and then adding each possible edge $(i,j), |i-j|\neq 1 \mod n$ with probability $\rho/n$ for some $\rho>0$ constant. In this paper we add i.i.d. exponential edge weights to this graph, and investigate typical distances in the corresponding random metric space given by the least weight paths between vertices. We show that typical distances grow as $\frac1\la \log n$ for a $\la>0$ and determine the distribution of smaller order terms in terms of limits of branching process random variables. We prove that the number of edges along the shortest weight path follows a Central Limit Theorem, and show that in a corresponding epidemic spread model the fraction of infected vertices follows a deterministic curve with a random shift.
\end{abstract}

\maketitle

\section{The model and main results}

\subsection{The Newman-Watts model}

The Newman-Watts small world model, often referred to as ``small world'' in short, is one of the first random graph models created to model real-life networks. It was introduced by Watts and Strogatz \cite{WatSto98}, and a simplifying modification was made by Newman and Watts \cite{NewWat99} later.
The Newman-Watts model consist of a cycle on $n$ vertices, each connected to the $k\ge1$ nearest vertices, and then extra shortcut edges are added in a similar fashion to the creation of the Erd\H{o}s-R\'enyi graph  \cite{ErdosRenyi60}: i.e., for each pair of not yet connected vertices, we connect them independently with probability $p$. 

The model has been studied from different aspects. Newman \emph{et al.} studied distances \cite{NewMooWat00, NewWat99scaleFPP} with simulations and mean-field approximation, as well as the threshold for a large outbreak of the spread of  non-deterministic epidemics \cite{NewMooCri00}. Barbour and Reinert treated typical distances rigorously. First, in \cite{BarReicont}, they studied a continuous circle with circumference $n$ instead of a cycle on  $n$ many vertices, and added $\Poi(n\rho/2)$ many 0-length shortcuts at locations chosen according to the uniform measure on the circle. Then, in \cite{BarReisdisc}, they studied the discrete model, with all edge lengths equal to 1. They showed that typical distances in both models scale as $\log n$.

Besides typical distances, the mixing time of simple random walk on the Newman-Watts model was also studied, i.e., the time when the distribution of the position of the walker gets close enough to the stationary distribution in total variation distance. 
Durrett \cite{Durrett07} showed that the order of the mixing time is between $(\log n)^2$ and $(\log n)^3$, then Addario-Berry and Lei \cite{NWmixingtime} proved that Durett's lower bound is sharp.

\subsection{Main results} \label{ss:mainresults}
We work on the Newman-Watts small world model \cite{NewWat99} with independent random edge weights:  we take a cycle $C_n$ on $n$ vertices, that we denote by  $[n]:=\{1,2,\dots, n\}$, and each edge $(i,j)\in [n], |i-j|=1\mod n$ is present. Then independently for each $i, j \in [n], |i-j|\neq1\mod n$ we add the edge $(i,j)$ with probability $\rho/n$ to form shortcut edges. The parameter $\rho$ is the asymptotic average number of shortcuts from a vertex. Conditioned on the edges of the resulting graph, we assign weights that are i.i.d.\ exponential random variables with mean $1$ to the edges. We denote the weight of edge $e$ by $X_e$. We write $\NWn(\rho)$ for a realization of this weighted random graph.

 We define the distance between two vertices in $\NWn(\rho)$ as the sum of weights along the shortest weight path connecting the two vertices. In this respect, the weighted graph with this distance function is a (non-Euclidean) random metric space. Further, interpreting the edge weights as time or cost, the distance between two vertices can also correspond to the time it takes for information to spread from one vertex to the other on the network, or it can model the cost of transmission between the two vertices.

We say that a sequence of events $\{\cE_n\}_{n\in \N}$ happens with high probability (w.h.p.) if $\lim_{n\to \infty}\P(\cE_n)=1$, that is, the probability that the event holds tends to $1$ as the size of the graph tends to infinity. We write $\Bin, \Poi, \Exp$ for binomial, $\Poisson$, and exponential distributions. For random variables $\{X_n\}_{n\in \N}, X$, we write $X_n\toindis X$ if $X_n$ tends to $X$ in distribution as $n\to \infty$. The \emph{moment generating function} of a random variable $X$ is the function $M_X(\vartheta):=\Ev[ \exp\{ -\vartheta X\}]$.

 Our first result is about typical distances in the weighted graph. Let $\Gamma_{ij}$ denote the set of all paths $\gamma$ in $\NWn(\rho)$ between two vertices $i,j \in [n]$. Then the weight of the shortest weight path is defined by
\begin{equation}\label{eq:P_ndef}
\CP_n(i,j) := \min_{\gamma\in\Gamma_{ij}} \sum_{e\in\gamma}X_e.
\end{equation}

\begin{theorem}[Typical distances]\label{thm:main} Let $U,V$ be two uniformly chosen vertices in $[n]$. Then, the distance $\CP_n(U,V)$ in $\NWn(\rho)$ with i.i.d.\  $\Exp(1)$ edge weights satisfies w.h.p.
\[\CP_n(U,V)-\frac{1}{\la} \log n \toindis -\frac{1}{\la} (\log W^U W^V +\Lambda + c), \]
where $\la$ is the largest root of the polynomial $p(x)=x^2+(1-\rho)x-2\rho$, $\Lambda$ is a standard Gumbel random variable, the random variables  $W^U, W^V$ are independent copies of the martingale limit of the multi-type branching process defined below in Section \ref{ss:bp}, and $c:=\log(1-\pi_R^2/2)-\log(\la(\la+1))$ with $\pi_R=2/(\la+2)$.
\end{theorem}

 Let us write $\gamma^\star=\gamma^\star(i,j)$ for the path that minimizes the weighted distance in \eqref{eq:P_ndef}. We call $\CH_n(U,V) := |\gamma^\star(U,V)|$ the \emph{hopcount}, i.e., the number of edges along the shortest-weight path between two uniformly chosen vertices.

\begin{theorem}[Central Limit Theorem for the hopcount] \label{thm:clt} Let $U,V$ be two uniformly chosen vertices in $[n]$. Then, the hopcount $\CH_n(U,V)$ in $\NWn(\rho)$ with i.i.d.\  $\Exp(1)$ edge weights satisfies w.h.p.
\[ \frac{\CH_n(U,V) - \frac{\la+1}{\la} \log n}{\sqrt{\frac{\la +1}{\la} \log n}} \toindis Z,  \]
where $Z$ is a standard normal random variable.
\end{theorem}

Our next result characterises the proportion of vertices within distance $t$ away from a uniformly chosen vertex $U$ as a function of $t$.
To put this result into perspective, note that we can model the spread of information starting from some \emph{source set} $I_0\subset [n]$ at time $t=0$ as follows: We assume that once a vertex $v$ receives the information at time $t$, it starts transmitting the information towards all its neighbors at rate $1$. Let us denote the vertices that are connected to $v$ by an edge by $\rH(v)$, then, for each $w \in \rH(v)$, $w$ receives the information from $v$ at time $t+X_{(v,w)}$. We further assume that transmission happens only after the first receipt of the information, that is, any consecutive receipts are ignored. If instead of the spread of information spread, we model the spread of a disease, this model is often called an $SI$-epidemic (susceptible-infected).

In the next theorem we consider this epidemic spread model from a single source $I_0=\{U\}$ on $\NWn(\rho)$ with i.i.d. $\Exp(1)$ transmission times.  We define
\be\label{eq:in-def} \rI_n(t,U):=\frac1n \sum_{i\in[n]}\ind\{i \text{ is infected before or at time } t\}= \frac1n \#\{ i: i \in [n], \Gamma_{U,j} \le t \},\ee
 the fraction of infected vertices  at time $t$ of the epidemic started from the vertex $U$.

\begin{theorem}[Epidemic curve] \label{thm:ec}
Let $U$ be a uniformly chosen vertex in $[n]$, and let us consider the epidemic spread with source $U$ and i.i.d.\  $\Exp(1)$ transmission times on $\NWn(\rho)$. Then, the proportion of infected individuals satisfies w.h.p.
\[ \rI_n(t + \tfrac1\lambda \log n,U) \toindis f(t + \tfrac1\lambda \log W_U), \]
where $f(t) = 1-\rM_{W_V}\left( x(t)\right)$, where $M_{W_V}(\cdot)$ is the moment generating function of $W_V$, and $x(t)= \left(1-\frac12 \pi_R^2\right)\e^{\lambda t}/ (\lambda(\lambda+1))$, with $\pi_R=2/(\la+2)$; and where  $W_U, W_V$, are the same random variables as in Theorem \ref{thm:main}.

\end{theorem}

\begin{remark} \normalfont
The intuitive message of Theorem \ref{thm:ec} is that a linear proportion of infected vertices can be observed after a time that is proportional to the logarithm of the size of the population. This time has a random shift given by $\tfrac1\lambda \log W_U$. Besides this random shift, the fraction of infected individuals follows a deterministic curve $f(\cdot)$: only the `position of the curve' on the time-axes is random. A bigger value of $W_U$ means that the local neighborhood of $U$ is ``dense'', and hence the spread is quick in the initial stages: indeed, a bigger value of $W_U$ shifts the function $f( t+ (\log W_U)/\la)$ more to the left on the time axes.
This phenomenon has been observed in real-life epidemics, see e.g. \cite{epiweb1, Tor04ec} for a characterisation of typical epidemic curve shapes. For individual epidemic curves, browse e.g. \cite{epiweb2}.
\end{remark}

The next proposition characterises the function $M_{W_V}(t)$ in the definition of the epidemic curve function $f(t)$ in Theorem \ref{thm:ec}.

\begin{proposition}[Functional equation system for the moment generating function]\label{funceq} The \\ moment generating function $M_{W_V}(\vartheta), \vartheta \in \R^+$ of the random variable $W_V$ satisfies the following functional equation system, with  $M_{W_V}(\vartheta):=M_{\WB}(\vartheta)$:
\be \ba M_{\WB} (\vartheta)
&= \left( \int_0^\infty M_{\WR} (\vartheta\e^{-\lambda x}) \e^{-x} \rd x \right)^2 \,\cdot\, \exp \!\left\{\rho \cdot \int_0^\infty \left(M_{\WR} (\vartheta\e^{-\lambda x}) -1\right) \e^{-x} \rd x \right\}, \\
M_{\WR} (\vartheta)
&= \int_0^\infty M_{\WR} (\vartheta\e^{-\lambda x}) \e^{-x} \rd x \,\cdot\, \exp\! \left\{\rho \cdot \int_0^\infty \left(M_{\WR} (\vartheta\e^{-\lambda x}) -1\right) \e^{-x} \rd x \right\}. \ea \ee
\end{proposition}
\begin{remark}\normalfont These functional equations and the fact that there exists a solution for all $\vartheta \in \R^+$ follow from the usual branching recursion of multi-type branching processes, that can be found e.g. in \cite{AthNey72BP}.
\end{remark}

\subsection{Related literature, comparison and context} \label{comp:br}
First passage percolation (FPP) was first introduced by Hammersley and Welsh \cite{HamWel65} to study spreading dynamics on lattices, in particular on $\Z^d, d\ge 2$. The intuitive idea behind the method is that one imagines water flowing at a constant rate through the (random) medium, the waterfront representing the spread. The model turned out to be able to capture the core idea of several other processes, such as weighted graph distances and epidemic spreads.

Janson \cite{Janson123} studied typical distances and the corresponding hopcount, flooding times as well as diameter of FPP on the complete graph. He showed that typical distances, the flooding time and diameter converge to 1, 2, and 3 times $\log n/n$, respectively, while the hopcount is of order $\log n$.

\emph{Universality class.}
 In a sequence of papers (e.g. \cite{BhaHofHoo11FPPer, BhaHofHoo12universality, BhaHogHoo10finitemean, HofHooZna05}) van der Hofstad \emph{et al.} investigated FPP on random graphs. Their aim was to determine \emph{universality classes} for the shortest path metric for  weighted random graphs without `extrinsic' geometry (e.g. the supercritical Erd\H{os}-R\'enyi random graph, the configuration model, or rank-$1$ inhomogeneous random graphs). They showed that typical distances and the hopcount scale as $\log n$, as long as the degree distribution has finite asymptotic variance and the edge weights are continuous on $[0,\infty)$. On the other hand, power-law degrees with infinite asymptotic variance drastically change the metric and there are several universality classes, compare \cite{HofHooZna05} with \cite{BhaHogHoo10finitemean}. In this respect, Theorems \ref{thm:main} and \ref{thm:clt} show that the presence of the circle does not modify the universality class of the model.

\emph{Comparison to the Erd\H{o}s-R\'enyi graph.} \label{comp:er}
Notice that the subgraph formed by shortcut edges is approximately an Erd\H{o}s-R\'enyi graph, with the difference that the presence of the cycle always makes $\NWn(\rho)$ connected and hence there is no subcritical or critical regime in $\NWn(\rho)$.
 Typical distances on the Erd\H os-R\'enyi graph with parameter $\rho/n$ and $\Exp(1)$ edge weights scale as $\log n/(\rho-1)$ \cite{BhaHofHoo11FPPer}, while for $\NWn(\rho)$ they scale as $(\log n)/\la$, with $\la=(\rho-1 + \sqrt{\rho^2 + 6\rho+1})/2> \rho-1$ for all $\rho\in \R$. This means that when $\rho>1$, the presence of the cycle makes typical distances shorter, and this appears already in the constant scaling factor of $\log n$. However, $\la(\rho)/\rho\to 1$ as $\rho\to \infty$ meaning that the effect of the cycle becomes more and more negligible as the number of shortcut edges grow.

\emph{Comparison to inhomogeneous random graphs.}
 Kolossv\'ary \emph{et al.} \cite{KolKom12} studied FPP on the inhomogeneous random graph model (IHRG), defined in \cite{BolJanRio07inhom}. In this model, vertices have types from a type space $S$,
 and conditioned on the types of the vertices, edges are present independently with probabilities that depend on the types. One can fine-tune the parameters of this model so that any finite neighborhood of a vertex in the $\NWn(\rho)$ model is similar to that of in the IHRG, that is, both of them can be modelled using \emph{the same} continuous time multi-type branching process. It would be natural to conjecture that typical distances are then the same in these two models. It turns out that this is almost but not entirely the case: the first order term $\la^{-1} \log n$, and the random variables $W_U, W_V$ are the same, but the additive constant $c$ in Theorem \ref{thm:main} is not: the geometry of the Newman-Watts model modifies how the two branching processes can connect to each other, which modifies the constant. Writing the main result in \cite{KolKom12} in the same form as the one in Theorem \ref{thm:main}, we obtain $c_\text{IHRG} = \log \left( (\rho+2)(2 \rho+\la^2) /( \rho(\la+2)^2\la(\la+1)\right)$.

\emph{The epidemic curve.} In \cite{BhaHofKom14epidemicFPP} Bhamidi \emph{et al.}\ pointed out the connection between FPP, typical distances,  and the epidemic curve by studying the epidemic spread on the configuration model with arbitrary continuous edge-weight distribution.  Earlier, \cite{BarRei13epidemiccurve} Barbour and Reinert investigated the epidemic curve on the Erd\H{o}s-R\'enyi random graph and on the configuration model with bounded degrees, where also possible other aspects such as contagious period of vertices or dependence of the transmission time distribution on the degrees might be present.

\emph{Possible future directions.}
In \cite{FPPcompetition11, BarHofKom14, DeiHof13competition} the competition of two spreading processes running on the same graph is investigated. This can be considered a competition between two epidemics, as well as the word-of-mouth marketing of two similar products. The results suggest that the outcome depends on the universality class of the model: in ultra-small worlds, one competitor only gets a negligible part of the vertices, while on regular graphs coexistence might be possible, i.e., both colors can paint a linear fraction of vertices. Studying competition on $\NWn(\rho)$ is an interesting and challenging future project.

\subsection{Structure of the paper}
In what follows, we prove Theorems \ref{thm:main}, \ref{thm:clt} and \ref{thm:ec}. The brief idea of the proof is the following: we choose two vertices uniformly at random, then we start to explore the neighbourhoods of these vertices in the graph in terms of the distance from these vertices (Section \ref{s:exploration}). We show that this procedure w.h.p.\ results in `shortest weight trees' ($\SWT$'s) that can be coupled to two independent copies of a continuous time multi-type branching process (CMBP). We then handle how these two shortest weight trees connect in the graph in Section \ref{s:connection} with the help of a Poisson approximation.  We provide the proof of Theorem \ref{thm:ec} about the epidemic curve in Section \ref{s:ec} based on our result on distances. Finally we prove the Central Limit Theorem for the hopcount in Section \ref{s:clt}, based on an indicator representation of the `generation of vertices' in the branching processes.

\section{Exploration process}\label{s:exploration}

To explore the neighborhood of a vertex, we use a modification of Dijkstra's algorithm.

Introduce the following notations: $\cN(t), \,\cA(t), \,\cU(t)$ denote the set of explored (dead), active (alive) and unexplored vertices at time $t$, respectively, and $\rN(t), \,\rA(t), \mathrm{U}(t)$ for the sizes of these sets.
The remaining lifetime of some vertex $w \in \cA(t)$ at time $t$ is denoted by $\cR_w(t)$, and means that $w$ will become explored exactly at time $t+R_w(t)$. The set of remaining lifetimes is $\cR_{\actset{t}}(t)$.
As before, $\rH(v)$ denotes the neighbors of a vertex $v$.

\subsection{The exploration process on an arbitrary weighted graph}
Let $i=1$.
The vertex from which we start the exploration process is denoted by $v_1$. We color $v_1$ blue and set the time as $t=T_1=0$. Evidently, we take
\[ \cN(0) = \{v_1\}, \quad
\cA(0) = \rH(v_1), \quad
\cU(0) = [n]\setminus\left(\{v_1\}\cup\rH(v_1)\right). \]
The remaining lifetimes are determined by the edge weights, i.e.
\[ \cR_{\actset{0}}(0) = \{\cR_w(0)=X_{(v_1,w)} \text{ for all } w \in \rH(v_1)\}. \]
We color the active vertices $w \in \rH(v_1)$ to have the same color as the edge $(v_1,w)$.

We work with induction from now on.
In each step, we increase $i$ by 1. We can construct the continuous time process in steps, namely, at the random times when we explore a new vertex.

Let $\tau_i=\min \left(\cR_{\actset{T_{i-1}}}(T_{i-1})\right)$, the minimum of remaining lifetimes. Then define $T_i:=T_{i-1}+\tau_i$, the time when we explore the next vertex. Nothing changes in the time interval $[T_{i-1},T_i)$, hence for any $t \in [T_{i-1},T_i)$,
\[  \cN(t):=\cN(T_{i-1}), \quad \cA(t):=\cA(T_{i-1}), \quad \cU(t):=\cU(T_{i-1}). \]
From all the remaining lifetimes, we subtract the time passed: for some $0\leq s\le \tau_i$,
\[ \cR_{\actset{T_{i-1}}}(T_{i-1}+s):=\cR_{\actset{T_{i-1}}}(T_{i-1})-s,\]
subtracted element-wise. 
At time $T_i$, the vertex (or all the vertices, if there is more than one such vertices) $v_i$ of which the remaining lifetime equals $0$, becomes explored and its neighbors become active. We shall refer to $v_i$ as the $i\ith$ explored vertex. We set
\[ \cN(T_i):=\cN(T_{i-1})\cup \{v_i\}, \quad
\cA(T_i):=(\cA(T_{i-1})\setminus \{v_i\})\cup \rH(v_i), \quad
\cU(T_i):=\cU(T_{i-1})\setminus \rH(v_i). \]
We refresh the set of remaining lifetimes:
\[ \cR_{\actset{T_i}}(T_i):=\cR_{\actset{T_{i-1}}}(T_i)\setminus\{\cR_{v_i}(T_i)\}\cup\{\cR_x(T_i) : x \in \rH(v_i)\} \]
where $\cR_x(T_i)=X_{(v_i,x)}$, the edge weight of $(v_i,x)$, and $x$ also gets the color of $(v_i,x)$. \\
On an arbitrary connected weighted graph, the exploration process can be continued until all vertices become explored.
Note that this algorithm builds the shortest weight tree $\SWT$ from the starting vertex. This tree will be modeled using the branching process.

\begin{remark} \label{remark:duplicates} \normalfont
The set of active vertices might contain several occurrences of a vertex, in case at least two neighbors of a vertex are explored already, see Figure \ref{fig::NW-coupling}.
\end{remark}

\begin{figure}[ht]\label{fig::NW}
\centering
\subfigure{\label{fig::NW-1}
\includegraphics[keepaspectratio, width=6cm]{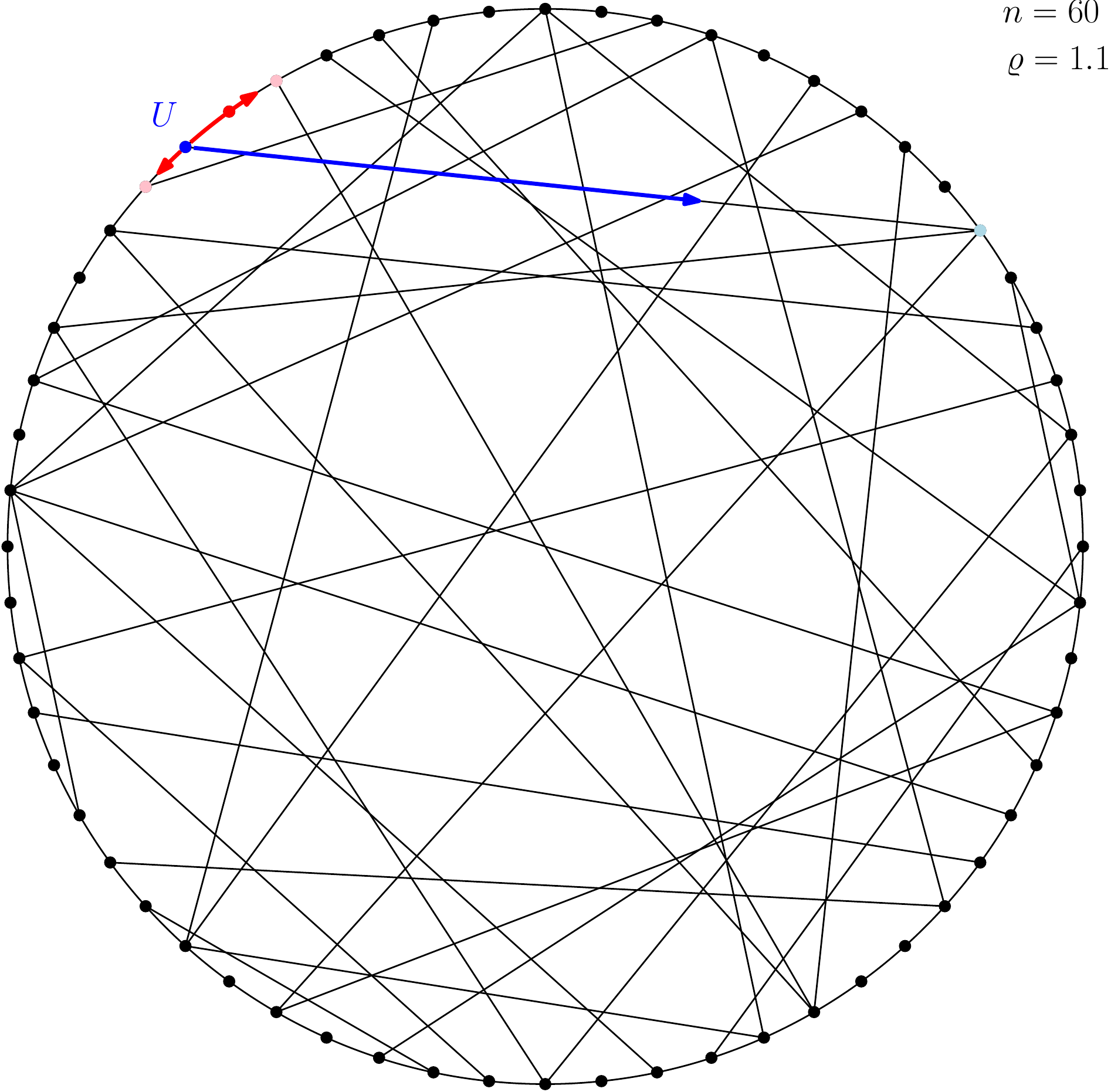}}
\subfigure{\label{fig::NW-2}
\includegraphics[keepaspectratio, width=6cm]{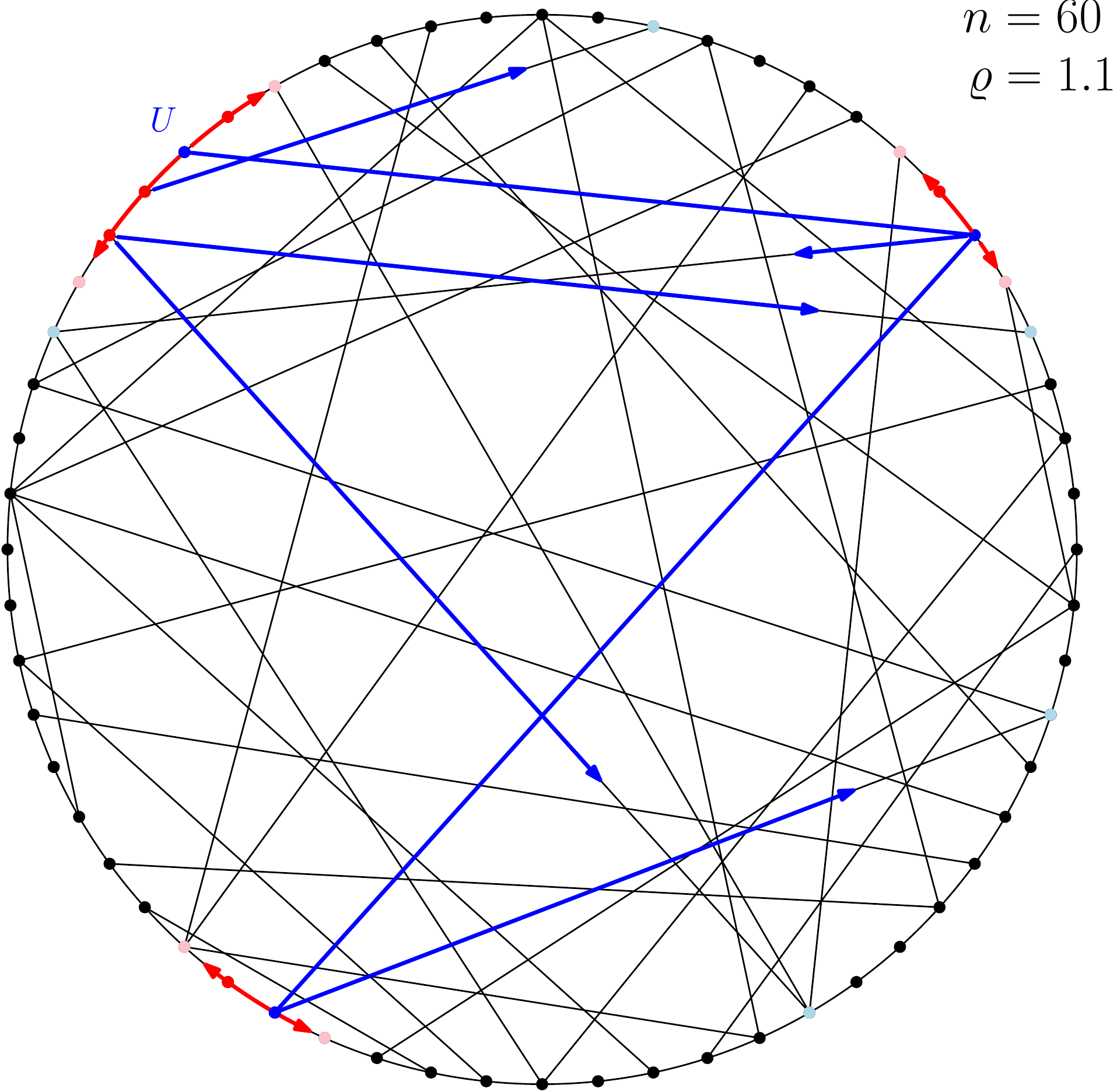}}
\caption{A realisation of the Newman-Watts model for $k=1$ and $\rho=1.1$ with $60$ vertices. On these two pictures, we illustrated the growing neighbourhood of a uniformly picked vertex. Circle edges are red and global edges are blue in the exploration. The edges that are partially red or blue are the ones that have an already explored vertex on one side while a not-yet explored (active) vertex on the other side.}\label{fig::NW-basic}
\end{figure}

\subsection{Exploration on the weighted Newman-Watts random graph}
 Note that when applying the exploration process above on a realization of $\NWn(\rho)$, we can reveal the presence of the edges and their weights in $\NWn(\rho)$ along the exploration process. In this respect, all the above quantities become random variables. Here we investigate the behaviour of this random exploration process.

Let us color the cycle-edges red and the shortcut-edges blue, and let us say that a vertex is red/blue if it is reached first via a red/blue edge during the exploration. We allow double (or more) occurrences of the same vertex in $[n]$ among the active vertices (with two remaining life-times along the two edges) in the exploration and also contradicting colorings. When such a vertex gets explored for the first time, it gets the color that corresponds to the remaining lifetime that became $0$ (and forget about the other colors).
 Below, adding the subscript $R$ or $B$ to any quantity corresponds to the same quantity restricted to only the red or blue vertices, respectively.

 When running the exploration process, we build a weighted tree along the process containing the edges that are used to discover the new vertices in the exploration (this is a tree since we do not explore vertices twice).
 This tree has root $v_1$, grows in time, and at any time $t$ it contains a vertex $v\in [n]$ precisely when $\cP(v_1, v)<t$. Let us denote the tree up to time $t$ by $\SWT^{v_1}(t)$.

\begin{claim}[Children] \label{lemma:children}
Suppose the vertex $v$ is being explored for the first time (i.e.,\ not "double-explored"). If $v$ is red, one new red and $\Binomial(n-3,\frac{\rho}{n})$ many new blue active vertices are born.
If $v$ is blue, two new red and $\Binomial(n-4,\frac{\rho}{n})$ many new blue active vertices are born.
The number of new blue active vertices is asymptotically $\mathrm{Poi}(\rho)$ in both cases. Further, at any time $t$, the elements of $\cR_{\actset{t}}(t)$ are i.i.d. $\Exp(1)$ random variables, and the next explored vertex is chosen uniformly over the set of active vertices.
\end{claim}
\begin{proof}
On a cycle there are two vertices neighboring a vertex, hence, if $v$ is red, then it has been reached from one of his neighbors. The other one is added to the new red active vertices. If $v$ is blue then it has been reached via a shortcut edge and hence both of its neighbors on the cycle are added to the new red active vertices. Since there are $\Bin(n-3, \rho/n)$ many shortcut edges from a vertex, this is also the distribution of new blue active vertices born when exploring a red vertex. For the exploration of a blue vertex, we reached this vertex via a blue edge, hence an additional $\Bin(n-4, \rho/n)$ new active blue vertices. Clearly, by the convergence of binomial to Poisson distribution, each vertex has asymptotically $\Poi(\rho)$ many blue neighbours. The second statement follows from the fact that the edge weights are i.i.d.\ exponential random variables, which has the memoryless property. Finally, note that at any time, $\cR_{\actset{t}}(t)$ consists of i.i.d.\ exponential random variables, and the algorithm takes the minimum of these. Clearly, the minimum of finite many absolutely continuous random variables is unique almost surely, and uniform over the indices.
\end{proof}

\subsection{Multi-type branching processes}
\label{ss:bp}
We define the following continuous time multi-type branching process (CMBP) that will correspond to the initial stages of $\SWT(t)$.

There are two particle types, red $(R)$ and blue $(B)$, and their lifetime is $\Exp(1)$, independent from everything else. Particles give birth upon their death. They leave behind offspring as in Claim \ref{lemma:children}: each particle has $\Poi(\rho)$ many blue offspring, red particles have one, while blue particles have two red children. Dead and alive particles will correspond to explored and active vertices, respectively. With this wording, for the number of alive and dead particles, we define
\begin{definition} \label{def:zt_nt}
We shall write $\mathbf A(t)=(\rA_R(t),\rA_B(t))$ for the number of alive particles of each type, $\rA(t)$ standing for the total number of alive particles.
Let $\rN(t)=\rN_R(t)+\rN_B(t)$,
where $\rN_q(t)$ means the number of dead particles of type $q=R,B$.
We assume the above quantities to be right-continuous.
Superscripts $(R),(B)$ refer to the process started with a single particle of the given type.
\end{definition}

The exploration process corresponds to the process started with a single blue-type particle, which dies immediately.

\subsubsection{Literature on multi-type branching processes}
Here we restate the necessary theorems from \cite{AthNey72BP} which we will use.

\begin{definition}[Mean matrix]
Let $\rM(t):= \rM_{r,q}(t) = \E [\rA_q^{(r)}(t)],\ (q,r=R,B)$ the mean matrix, where $\rA_q^{(r)}(t)$ is as defined above in Definition \ref{def:zt_nt}.
\end{definition}
It is not hard to see that $\rM(t)$ satisfies the semigroup property $\rM(t+s)=\rM(t)\rM(s)$ and the continuity condition $\lim_{t \to 0}\rM(t)=\rI$, where $\rI$ denotes the identity matrix. As a result, we have:
\begin{theorem}[Athreya-Ney]\label{thm::infinitesimal}
There exists an infinitesimal generator matrix $\rQ$ such that $\rM(t)=\e^{\rQ t}$, where $Q_{r,q}=a_r \E[D^{(r)}_q]-\delta_{r,q}$.
Here, $a_r$ is the rate of dying for a particle of type $r$, (i.e., the parameter of its exponential lifetime), $D$ is the number of offspring with the same sub-end superscript conventions as in Definition \ref{def:zt_nt}, and $\delta_{r,q}=\ind_{\{r=q\}}$ (i.e., $\delta_{r,q}=1$ if and only if $r=q$).
\end{theorem}
In our case,
\[ \rQ =\left(
\begin{array}{cc}
0 & \rho \\
2 & \rho-1
\end{array}
\right)\]

\paragraph*{\textit{Eigenvalues and eigenvectors of the $\rQ$ matrix}} Using the characteristic polynomial, for $\rho\geq0$, the maximal eigenvalue $\la$ and the second eigenvalue $\la_2$ is given by 
\begin{equation} \label{eqation:lambda}
\lambda=\frac{\rho-1+\sqrt{\rho^2+6\rho+1}}{2}, \quad \la_2=\frac{\rho-1-\sqrt{\rho^2+6\rho+1}}{2}.
\end{equation}
The normalised left eigenvector $\pi$ that satisfies $\pi \rQ = \la \pi$ gives the stationary type-distribution:
\begin{equation} \label{eq:pi}
\bpi = \left(\pi_R, \pi_B\right)=\left(\frac2{\lambda+2}, \frac{\lambda}{\lambda+2}\right).
\end{equation}
We denote the right (column) eigenvector of $Q$ by $\bu$ and normalize it so that $\bpi \bu = 1$.
For later use, without computing, we denote by $\bv_2$ and $\bu_2$ the left (row) and right (column) eigenvector of $\rQ$ belonging to the eigenvalue $\lambda_2$. The most important theorem for our purposes is that the CMBP grows exponentially with rate $\lambda$ (the so-called Malthusian parameter), more precisely,
\begin{theorem}[\cite{AthNey72BP}]\label{theorem:actives} With the notation as above, almost surely,
\[ \lim_{t\to\infty} \mathbf A(t)\e^{-\lambda t} = W \bpi \]
where $W$ is a nonnegative random variable, the almost sure martingale limit of $W_t := \mathbf A(t)\bu\,\e^{-\lambda t}$. Further, $W>0$ almost surely on the event of non-extinction.
\end{theorem}

\begin{theorem}[\cite{AthNey72BP}] \label{thm:athreyaney}
Define $T_m$, the $m\ith$ split time, as the time of the $m\ith$ death in the branching process. (We assume $T_1=0$ for the death of the root.) On the event $\{W>0\}$,
\begin{enumeratei}
\item For each $q \in (R,B)$, $\lim_{m\to\infty} \rN_q(T_m)/\rN(T_m) = \lim_{m\to\infty} \rN_q(T_m)/m \eqas \pi_q$
\item $\lim_{m\to\infty} m\e^{-\la T_m} \eqas \frac1\la W$
\end{enumeratei}
\end{theorem}
\begin{corollary} \label{corollary:dead} For the vector of dead particles $\mathbf N(t)=(N_R(t), N_B(t))$,
\[\mathbf N(t)\e^{-\lambda t} \toas \frac1\lambda W \bpi.\]
\end{corollary}
\begin{proof}[Proof of Theorems \ref{thm::infinitesimal}, \ref{theorem:actives}, \ref{thm:athreyaney} and Corollary \ref{corollary:dead}] The proofs can be found in \cite[Chapter V.7]{AthNey72BP}.
%
\end{proof}
Throughout the next sections, we develop error bounds on the coupling between the branching process and the exploration process on the graph. For convenience, we introduce
\begin{equation} \label{def:tn}
t_n:=\frac{1}{2\la} \log n,
\end{equation}
the times we will observe the branching and exploration processes at, as well as
\be \label{def:wn}
W^{\scriptscriptstyle{(n)}}:=e^{-\la t_n} \rA(t_n), \quad \text{with} \quad W^{\scriptscriptstyle{(n)}}\toas W,
\ee
the approximations of the martingale limit $W$ at the times $t_n$. Note that in our case, extinction can never occur, hence almost surely $W>0$.


\subsection{Labeling, coupling, error terms}
In this section we develop a coupling between the CMBP discussed in the previous section and $\SWT(t)$, the exploration process on $\NWn(\rho)$.
\subsubsection*{Error bound on coupling the offspring} 
The CMBP is defined with $\Poi(\rho)$ blue offspring distribution, while in the exploration process a vertex has $\Bin(n-3,\rho/n)$ or $\Bin(n-4,\rho/n)$) many blue children.
Let $Y\sim\Poi(\rho)$ and $X\sim\Bin(n,\rho/n)$. By the usual coupling of binomial and Poisson random variables, $\P(X\neq Y)\leq \frac{\rho^2}{n}$. Let $Z\sim\Bin(n-3,\rho/n)$, $V\sim\Bin(3,\rho/n)$ independent.
Then $Z\eqindis X-V$, and under the usual coupling
\[ \P(Z\neq Y) \leq \P(X\neq Y)+\P(V\neq0)= \frac{\rho^2}{n}+\frac{3\rho}n + o(1/n^2). \]
For the blue offspring of a blue vertex $\hat{Z}\sim\Bin(n-4,\rho/n)$, by similar arguments $\P(\hat{Z}\neq Y) \leq \frac{\rho^2}{n}+\frac{4\rho}n + o(1/n^2)$ holds. Taking maximum and using union bound, the probability that up to $k$ steps, at least one particle has different number of blue offspring in the exploration process and the Poisson branching process, is at most $k(\rho^2+4\rho)/n$.

\subsubsection{Labeling and thinning} \label{sss:thin} We relate the CMBP to the exploration process on $\NWn(\rho)$ through the labeling of the earlier. Below, everything must be interpreted modulo $n$.
\begin{enumeratei}
\item The root is labeled $u$, the source of the exploration process. $u$ can be $U$, a uniformly chosen vertex in $[n]$.
\item Every other particle gets a label when it is born.
\item We distinguish "left type" and "right type" red children. Left type red particles have a left type red child, right type red particles have a right type red child, blue particles have a red child of both types.
\item A left type red child of $v$ gets label $v-1$, a right type red child of $v$ is labeled $v+1$.
\item The blue children of $v$ get a set of labels uniformly chosen from $[n]$.
\end{enumeratei}

\begin{lemma}\label{cor:thinning-prob-bound} We say that the labeling fails if two \emph{explored} vertices share the same label (this still allows for several occurrences of the same label in the active set). The probability that the labeling fails at the $i\ith$ split is at most $2i/n$.
\end{lemma}
\begin{proof}
The labeling fails at the $i\ith$ split if the splitting particle has a label that is already taken by an explored vertex. We distinguish two cases.

\paragraph{\textit{When a blue particle splits}}
Since the label of a blue particle is chosen uniformly in $[n]$, and there are at most $i-1$ dead labels already, the probability that we choose from this set is $(i-1)/n$.

\paragraph{\textit{When a red particle splits}} \label{p:redsplit}
Note that the labeling procedure ensures that whenever a blue particle $v$ is explored, it starts a growing (possibly asymmetric) red interval of red vertices around it. A red vertex, upon dying, extends this interval in one direction (if it is left type, then towards the left).
 Note that the original vertex $v$ in this interval had a uniformly chosen label in $[n]$. Let us denote the position of the $k\ith$ explored blue vertex by $c_k$, and write $l_k(T_i)$ and $r_k(T_i)$ for the number of explored red vertices to the left and to the right of $c_k$ after the $i\ith$ split, $i\ge k$. Finally, we denote the whole interval of explored vertices around $c_k$ after the $i\ith$ split by $I_k(T_i)$. Recall that the process is by definition right-continuous.

In this setting, the label of a red vertex that is just being explored can coincide with the label of an already explored red vertex if and only if two intervals `grow into each other' at the $i\ith$ split. 
Denote by $I^*$ the interval that grows at the $i\ith$ split, write $c^*, r^*(T_{i-1}), l^*(T_{i-1})$ for the location of its blue vertex, right and left length, respectively.
 Then, $I^\star$ grows into another interval $I_k$ if and only if $c_k$,  the location of the  blue vertex in $I_k$, is at position $c^*-l^*(T_{i-1})-r_k(T_{i-1})-1$ or is at position $c^*+r^*(T_{i-1})+l_k(T_{i-1})+1$. (The first case means that the furthest explored red vertex on the right of $I_k$ was a red active child of the furthest explored left vertex in $I^*$).
Since the location of $c_k$ is uniform in $[n]$,
\[ \P(I^*(T_{i-1}) \cap I_k(T_{i-1})= \varnothing , I^*(T_{i}) \cap I_k(T_{i})\neq \varnothing)= \frac2n.\]
 Note that there are exactly as many intervals as blue explored vertices (at either $T_{i-1}$ or $T_i$, since the $i\ith$ explored vertex $v_i$ must be red). Let the event $E_i=\{v_i$ is red and its label is already used$\}$. Hence,
\[ \P(E_{i}) \leq \sum_{k=1}^{\rN_B(T_{i-1})-1} \frac2n = \frac2n \left(\rN_B(T_{i-1})-1\right)\le \frac{2i}{n},\]
since there are at most $i$ blue explored vertices.
Note that the proof also applies when the new red explored vertex coincides with a formerly explored blue one, in case $l_k(T_{i-1})=0$ or $r_k(T_{i-1})=0$. Hence, the statement of the lemma follows.
\end{proof}

In $\NWn(\rho)$, the shortest path $(u,v)$ through $x$ necessarily uses the shortest path between $(u,x)$. As a result, in the CMBP, we also do not need later occurrences of the label $x$. Hence, we mark the second (or any later) occurrence of a label \emph{thinned}, and all its descendants \emph{ghosts}. We move towards bounding the proportion of ghosts among active individuals to carry on with the CMBP approximation. To determine whether a vertex is a ghost, we need knowledge about its ancestors.

\subsubsection{\textit{Ancestral line}} \label{p:ancestralline}
We approach the problem of ghost actives with the help of the ancestral line. We define the ancestral line $\rAL(w)$ of a vertex $w$ as the chain of particles leading to $w$ from the root, including the root and $w$ itself. Then an alive particle is a ghost if and only if at least one of its ancestors is thinned.
The ancestral line was introduced by B\"uhler in \cite{Buhler71, Buhler72gen} with the following observation: for each time interval $[T_{k},T_{k+1})$ we can allocate a unique particle on the ancestral line that was active in the interval $[T_{k},T_{k+1})$.
For the following observations, we condition on $\{D_i, i=1,...,k\}$, where $D_i$ is the total number of offspring of the $i\ith$ splitting particle. Denote by $G_k$ the generation of a uniformly chosen alive (active) particle $W$ after the $k\ith$ split. Then $G_k=L_1+L_2+...+L_k$, where the indicators $L_i$ are conditionally independent and $L_i=1$ if and only if the ancestor of $W$ that was alive in the time interval $[T_i,T_{i+1})$ was newborn (born at $T_i$). (A rewording of the indicators $L_i$ is as follows: $L_i=1$ if and only if the $i\ith$ splitting particle is in $\rAL(W)$.)

Since $W$ is chosen uniformly, and at each split the individual to split is also chosen uniformly among the currently active individuals, each one of these active individuals is equally likely to be an ancestor of $V$. Further, in the interval $[T_i,T_{i+1})$, $D_i$ many particles are newborn, and $S_i$ many are alive, which yields the probability $\P(L_i=1| D_i, i=1,...,k)=D_i/S_i$, see the discussion at the beginning of \cite[Section 2.A]{Buhler72gen}. We arrive to the following corollary:

\begin{corollary}\label{cor:ancestral-line}
The probability of the $i\ith$ dying particle being an ancestor of $W$, a uniformly chosen active vertex after the $k\ith$ split:
\[\P(v_i \in \rAL(W)|D_i, i=1,...,k )=\P(L_i=1)=\frac{D_i}{S_i}.\]
\end{corollary}

\paragraph*{\textit{Expected proportion of thinned actives}} Let us combine Corollary \ref{cor:ancestral-line} and Lemma \ref{cor:thinning-prob-bound}.
To be able to do so, we need the following lemma. We will provide its proof later on.
\begin{lemma} \label{lemma:si}
For every $\ve>0$, there exists a positive integer-valued random variable $K=K(\ve)$ so that $K$ is always finite and for every $i>K,\; S_i = i \lambda (1+o(i^{-1/2+\ve}))$ holds.
\end{lemma}
Recall that $t_n=(\log n) / 2 \la$, and it was chosen such that the number of active vertices is of order $\sqrt{n}$, and that $\cA(t), \rA(t), \cN(t), \rN(t)$ denotes the set and number of active and dead individuals in the CMBP at time $t$, respectively.
\begin{lemma} \label{lemma:ghosts}
Let $\cA_G(t)=\{w \in \cA(t) : w \text{ is a ghost}\}$ the set of ghost active vertices at time $t$ and $\rA_G(t)$ its size. For every fixed $s\in \R$, the proportion $\rA_G(t_n+s)/\rA(t_n+s)$ tends to 0 in probability as $n$ tends to infinity.
\end{lemma}
\begin{proof}[Proof of Lemma \ref{lemma:ghosts}]
The proportion $\rA_G(t)/\rA(t)=\P(W \in \cA_G(t))$, where $W$ is uniform over $\cA(t)$, i.e., uniformly chosen active individual. Recall that $v_i$ is the particle that dies at $T_i$. For an event $E$, let us write $\Pv_k(E):=\Pv(E| D_i, i=1, \dots, k)$.
Using these notation and Corollary \ref{cor:ancestral-line} for the representation of the ancestral line of $V \in \cA(t)$, we can write 
\[ \P_k(W \in \cA_G(t))
 \leq \sum_{i=1}^{\rN(t)} \P_k(v_i \in \rAL(W)\text{ and } v_i \text{ is thinned}), \]
Since the labeling is independent of the family tree,
\be \label{eq:ghostprob}
\P_k(W \in \cA_G(t))\leq \sum_{i=1}^{\rN(t)} \P_k\left(v_i \in \rAL(W)\right) \cdot \P_k\left(v_i\text{ is thinned}\right)  \le \sum_{i=1}^{\rN(t)} \frac{D_i}{S_i} \frac{2i}n.
\ee
 We apply Lemma \ref{lemma:si} by splitting the sum for parts up to $K$ and above, use $D_i<S_i$ for $i\le K$:
\[ \ba \P(W \in \cA_G(t)) &\leq \sum_{i=1}^{K} \frac{2i}{n} + \sum_{i=K+1}^{\rN(t)} \frac{D_i}{\lambda(1+o(i^{-1/2+\ve}))n} \\
&\le \frac{K^2}{n} + 2\frac{\sum_{i=K+1}^{\rN(t)}D_i}{\lambda n} < \frac{K^2}{n} + 2\frac{\rA(t)+ \rN(t)}{\lambda n} \ea \]
where we used that all particles are either active or dead in the process and with a  possible modification of $K$, we can have $(1+o(i^{-1/2+\ve})> 1/2$ for all $i>K$.
Next, we can use Corollary \ref{corollary:dead} and Theorem \ref{theorem:actives}, which gives
that $\rN(t)+\rA(t)=\e^{\lambda t}(\frac1\lambda+1) W^{\scriptscriptstyle{(n)}}(1+o(1))$.
 Hence
\[\rA_G(t_n+s)/\rA(t_n+s) = \P(W \in \cA_G(t_n+s))\leq K^2/n + 2\e^{\lambda (t_n+s)}\frac{\lambda+1}{\lambda^2 n}W^{\scriptscriptstyle{(n)}}(1+o(1)).\]
Setting $t_n=\log n/ (2\la)$,
 the right hand side tends to $0$ as $n \to \infty$, since $W^{\scriptscriptstyle{(n)}}\to W$ and $K$ is a tight random variable (does not depend on $n$).
\end{proof}
Let us now return to the proof of Lemma \ref{lemma:si}. This lemma follows from \cite[Theorem 1, Theorem 2]{Asm77LIL}. Here, we restate \cite[Theorem 1]{Asm77LIL} using our notations and for a special case, where each eigenvalue has multiplicity 1. This is sufficient for our purposes and easier than the general case.
\begin{theorem}[Asmussen, \cite{Asm77LIL}] \label{theorem:asmussen1}
Let $\bZ_n$ be the number of individuals in the $n\ith$ generation of a (discreet time) supercritical multi type Galton-Watson process, with dominant eigenvalue $\lambda$, the corresponding left and right eigenvector $\bv$ and $\bu$. For any other eigenvalue $\nu$, $\bv_\nu$ and $\bu_\nu$ denote the left and right eigenvectors belonging to $\nu$. \\
 For an arbitrary vector $\bba \in \R^p$ with the property $\bv \cdot \bba=0$ define
\be\label{eq::defmu}
\mu:=\sup\{\nu : \bv_\nu \bba \neq 0\}, \quad
\sigma^2 := \lim_{n\to\infty} \frac{|\bv| \Var (\bZ_n \bba)}{\lambda^n}  \ee
If $\mu^2<\lambda$, then with $C_n = (2\sigma^2 \bZ_n \bu \log n)^{1/2}$
\[\ba \liminf_{n\to\infty} \frac{\bZ_n \bba}{C_n} &= -1 \quad \text{and} \quad \limsup_{n\to\infty} \frac{\bZ_n \bba}{C_n} &= 1. \\
  \ea\]
\end{theorem}
We also restate \cite[Theorem 2]{Asm77LIL} without change.
\begin{theorem}[Asmussen 2., \cite{Asm77LIL}] \label{theorem:asmussen2}
Replacing $\mathbf Z_n$ with $\mathbf A(t), t \in [0,\infty)$, Theorem \ref{theorem:asmussen1} remains valid for any supercritical irreducible multi-type Markov branching process.
\end{theorem}
\begin{proof}[Proof of Lemma \ref{lemma:si}]
We use the previous two theorems for the 2-type branching process defined in Section \ref{ss:bp}.
Since $\pi$ and $\bv_2$ are linearly independent, for any $\bba\neq (0,0)$ with $\pi \bba = 0$, necessarily $\bv_2 \bba \neq 0$, which implies $\mu = \lambda_2$ in \eqref{eq::defmu}.
The eigenvalues of the mean matrix $\rM(t)$ are $\e^{\lambda t}$ and $\e^{\lambda_2 t}$. The condition $\mu^2<\la$ in Theorem \ref{theorem:asmussen1} is then equivalent to $2\la_2<\la$ which follows from the nonnegativity of $\rho$, through simple algebraic computations, see \eqref{eqation:lambda}.
The asymptotic variance $\sigma^2$ and $C_t$ in this case becomes:
\[ \sigma^2 = \lim_{t\to\infty} \pi \Var (\mathbf A(t) \bba) \e^{-\lambda t}, \quad
C_t = (2 \sigma^2 \mathbf A(t) \bu \log t)^{1/2} \]
This implies that the theorem rewrites to
\[
\limsup_{t\to\infty}\frac{\mathbf A(t) \bba}{C_t} =1 \text{ and } \liminf_{t\to\infty}\frac{\mathbf A(t) \bba}{C_t} =-1.  \]
Applying this for the split time $T_i$, we get that there is only a finite number of indexes $i$ such that $\left|\mathbf A(T_i) \bba /C_{T_i}\right| > 2$. Let the maximum of these indexes be $K$, a random variable. Since $T_i - \log i/\lambda $ has an almost sure limit by Theorem \ref{thm:athreyaney}, $T_i$ is of order $\log i$. This implies that $C_{T_i}$ is of order $(i\log\log i)^{1/2}$, and by definition of the almost sure convergence,  $C_{T_i}$ exceeds $i^{1/2+\ve}$ only finitely many times for every $\ve>0$.

Since $\E[\mathbf A (t) \bba] = 0$ if and only if $\pi \bba = 0$, we can apply the theorem for the centered version $S_i^c:=S_i-\E S_i$. Then for $i>K$, $|S_i^c|\leq C_{T_i}$. The fluctuation is of smaller order then $S_i;$ itself, which means we can indeed write $S_i=i\la(1+o(i^{-1/2+\ve}))$. For more detail on this, see the proof of \cite[Corollary 3.16]{Jan04functional}.
\end{proof}

\subsubsection{The number of multiple active and active-explored labels}
Recall that both in the exploration process as well as in the branching process there might be multiple occurrences of active vertices, see Remark \ref{remark:duplicates}, as thinning only prevents multiple explored labels.
\begin{figure}[ht]
\centering
\includegraphics[keepaspectratio, width=6cm]{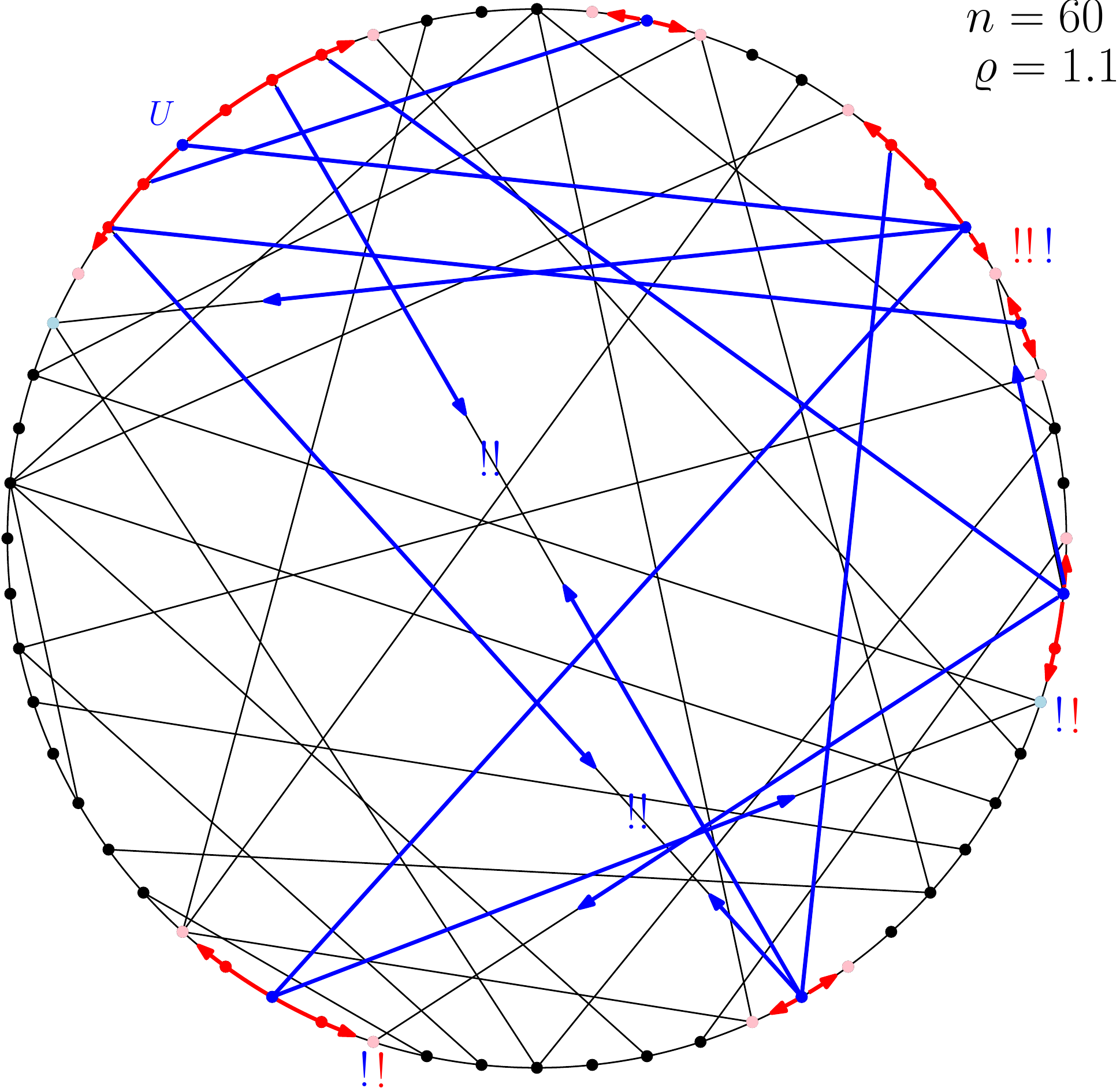}
\caption{We indicated the growing neighbourhood of a uniformly picked vertex in the exploration process. Exclamation marks indicate `bad events' for the coupling to a branching process: the vertices at the endpoint of edges (indicated along the edge) with two blue exclamation marks are vertices that are blue active and have been already explored as well. The vertex with two red and one blue exclamation mark is twice red active and once blue active at the same time.}\label{fig::NW-coupling}
\end{figure}
Later we want to use that the number of different active labels that are not ghossts at $T_i$ is approximately the same as $S_i$, i.e.,\ there are not many multiple occurrences. In Lemma \ref{lemma:ghosts} we have seen that the proportion of ghosts is negligible on the time scale $t_n$, but we still have to deal with labels that are multiply active, or are explored and active at the same time. We will discuss these issues in the following five cases:
\begin{enumeraten}
\item A blue active vertex has been already explored.
\item A red active vertex has been already explored.
\item A blue active vertex is also red active.
\item A vertex is double red active.
\item A vertex is double blue active.
\end{enumeraten}
We will denote by $p_\alpha(t)$ the probability that a uniformly chosen active vertex falls under case $\alpha=1,\dots, 5$ at time $t$, which is the same as the proportion of vertices falling under case $\alpha$ among all active vertices. 

\paragraph*{Case 1. \textit{Blue active being already explored}} At time $t$, there are at most $\rN(t)$ explored labels that are not thinned. Under the condition that the active vertex is blue, its label is chosen uniformly over $[n]$, so the probability that this label has been already explored is at most $\rN(t)/n$. Substitute $\rN(t_n+s)$ from Corollary \ref{corollary:dead}, then for $t_n+s=\frac{1}{2\la}\log n+s$,
\be\label{eq::p1} p_1(t_n+s)\le\P(v \text{ is already explored } \vert \ v \text{ is blue})=\rN(t_n+s)/n. \ee

\paragraph*{Case 2. \textit{Red active being already explored}} This case can be treated similarly as the thinning of red vertices, so we also use the notation there. A label of the red active vertex is explored if and only if two intervals are about to grow into each other: the furthest explored red vertices in both intervals are neighbors. We call these intervals neighbors. Then, for two neighboring intervals, the active vertices at the end of each interval are explored in the other interval. Let $I_k$ and $I_j, 1\leq k<j\leq \rN_B(t)$ intervals with blue particles with label $c_k$ and $c_j$ respectively. Conditioned on $c_k, l_k(t), r_j(t)$, there are two possibilities so that $I_k$ and $I_j$ are neighbors: $c_j=c_i+r_i+l_j+1$ or $c_j=c_i-l_i-r_j-1$.  Thus for each pair of indices the probability of the intervals being neighbors is $2/n$ (these are not independent, but expectation is linear). Summing up for all pair of indexes and dividing by the number of all red actives gives the proportion of \textit{case 2} red actives among all red actives.
\be\label{eq::p2} \ba p_2(t_n+s) \le \frac{1}{\rA_B(t_n+s)}\displaystyle \sum_{1\leq i<j\leq \rN_B(t_n+s) }\frac2n
= \frac{{\rN_B(t_n+s) \choose 2} \cdot 2/n}{2\rN_B(t_n+s)} \le \frac{\rN_B(t_n+s)}{2n}. \ea \ee

\paragraph*{Case 3. \textit{Blue active being red active}} Using that the labels of blue vertices are chosen uniformly,
\be\label{eq::p3} \P(v \text{ is red and blue active\,}) = \rA_R(t_n+s)/n. \ee
\paragraph*{Case 4. \textit{Multiple red active vertices}} This case is similar to Case 2. A vertex $v$ can be red active twice if the two intervals that it belongs to are "almost neighbors", that is, both have $v$ as an active vertex on one of their ends. ($v$ is the only vertex separating them.) Conditioning on the location of one of the intervals, the blue vertex in the other interval can be at $2$ different locations, hence
\be \label{eq::p4} p_4(t_n+s) = p_2(t_n+s). \ee


\paragraph*{Case 5. \textit{Multiple blue active vertices}}
 Again, the label of a blue vertex is chosen uniformly at random, hence the probability that the label of an active blue vertex $v$ coincides with another active blue label is at most $\rA_B(t)/n$. Hence
 \be\label{eq::p5} p_5(t_n+s) \le \frac{1}{\rA_B(t_n+s)} \sum_{v\in \cA_B(t_n+s)} \rA_B(t_n+s)/n = \rA_B(t_n+s)/n.  \ee
\begin{corollary} \label{cor:effectivesize}
Define $\rA_e(t)$ the \textit{effective} size of the active set as follows: we subtract from $\rA(t)$ the number of ghosts, already explored and multiple active labels, to get the number of different labels in $\cA(t)$. Then
\[\rA_e(t_n+s)/\rA(t_n+s)\  {\stackrel \P \longrightarrow}\  1.\]
\end{corollary}
\begin{proof}
By the previous arguments, a lower bound can be given if we subtract the individual probabilities for red and blue vertices to be deleted (note that this is a crude bound since we do not weight it with the proportion of red and blue active labels):
\[  \rA_e(t_n+s)/\rA(t_n+s) \geq 1- \sum_{i=1}^5 p_{i}(t_n+s) = 1- \frac{2\rN(t_n+s)+ \rA(t_n+s)}{n},\]
where we summed up the rhs of \eqref{eq::p1}, \eqref{eq::p2}, \eqref{eq::p3}, \eqref{eq::p4}, \eqref{eq::p5} to obtain the rhs.
Now we can use that $t_n+s=\log n/(2\la) +s $, use $\rN(t)$ from Corollary \ref{corollary:dead}, and Theorem \ref{theorem:actives} to get
\[  \rA_e(t_n+s)/\rA(t_n+s) \geq 1- \frac{\la+2}{\la} \e^{\la s} W^{\scriptscriptstyle{(n)}}(1+o(1))/\sqrt{n} ,\]
which tends to $1$ since $W^{\scriptscriptstyle{(n)}}\to W$ a.s. by \eqref{def:wn} and Theorem \ref{theorem:actives}.
\end{proof}
The conclusion of this section is summarized in the following corollary.
\begin{corollary}
Fix $n\ge 1$ and $\rho>0$. Consider the
thinned CMBP with label $u$ for the root. Then, there is a coupling of shortest weight tree  $\SWT^u(t)$
in
$\NWn(\rho)$ to the evolution of the thinned CMBP as long as $t\le t_n +M$ for some arbitrary large $M\in \mathbb R$. Further, the set of active vertices in the thinned CMBP can be approximated by the set of labeled active vertices in the the original CMBP in the sense that the proportion of the different labels among the actives over the total active vertices tends to zero as $n\to \infty$, in the sense of Corollary \ref{cor:effectivesize}.
\end{corollary}

\section{Connection process}\label{s:connection}

Now that we have a good approximation of the shortest weight tree ($\SWT$) started from a vertex, it provides us a method to observe the shortest weight path between two vertices. Let us give a raw sketch of this method before moving into the details. The previous section provides us with a coupling of a CMBP and the $\SWT$ as long as the total number of vertices is of order $\sqrt{n}$ in the $\SWT$. To find the shortest weight path between vertices $U$ and $V$, we grow the shortest weight trees from one of the vertices ($\SWT^U$) until time $t_n$ (the size is then of order $\sqrt{n}$). Then, conditioned on the presence of $\SWT^U(t_n)$, we grow $\SWT^V$ and see when is the first time that these two trees intersect. The shortest weight path is determined by the first intersection of the \textit{explored} set of vertices in the two processes. Note that a vertex $w$ in the active set of vertices in $SWT^U(t_n)$ is at distance $t_n+ R_w(t_n)$ from $U$.
 Since we have a good bound on the effective size of the set of active vertices, it turns out to be easier to look at the times when the first few active vertices in $\SWT^U(t_n)$ become explored in $\SWT^V$, and then minimize over the total distance of the path formed by connecting $U$ and $V$ via this vertex. This is what we shall carry out now rigorously.

\begin{definition}[Collision and connection] \label{def:collision-connection}
We grow $\SWT^U$, the shortest weight tree of $U$ until time $t_n=\frac1{2\lambda}\log n$, and then fix it. Then we grow $\SWT^V$ until time $t_n+M$, for some large $M\in \R$, conditioned on the presence of $\SWT^U(t_n)$. We say that a collision happens at time $t_n+s$ when an active vertex in $\SWT^U(t_n)$ becomes \textit{explored} in $\SWT^V$ at time $t_n+s$. Denote the set of collision times by the point process $(t_n+ P_i)_{i\in \N}$.\footnote{We will see later that a.s. there is a first collision time. Hence, indexing by $i\in \N$ makes sense.} If a collision happens at vertex $x_i$ at time $t_n+P_i$, this determines a path between $U$ and $V$ with length $2t_n+P_i+\cR^{U}_{x_i}(t_n)$, where $\cR^U_{x_i}(t_n)$ is the remaining lifetime of $x_i$ in $\SWT^U(t_n)$. Then the length of the shortest weight path is given by
\[   \min_{i\in \N} \left(2t_n+P_i+\cR^{U}_{x_i}(t_n+P_i) \right)\] among all collision events.
\end{definition}
\begin{figure}[ht]
\centering
\includegraphics[keepaspectratio, width=6cm]{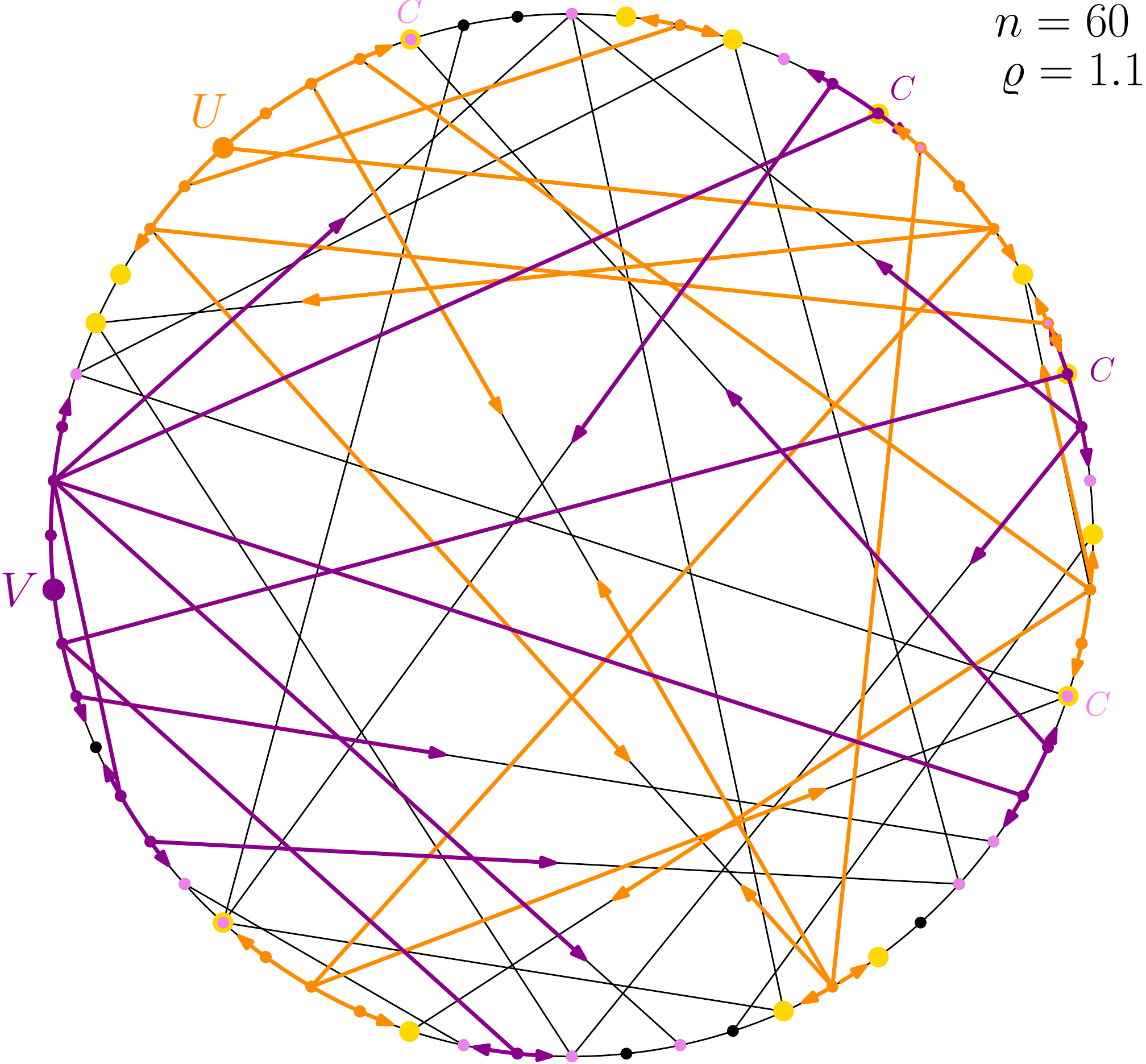}
\caption{We indicated the growing neighbourhood of two uniformly picked vertices in the exploration process, with purple and yellow colors, respectively. The letters `C' on the picture show that a collision event happens at the given vertex. Notice that all these collision events have a remaining edge-length yet to be covered in the exploration process of $U$, i.e., the vertex is active in $\SWT^U$ and explored in $\SWT^V$.}\label{fig::NW-coupling}
\end{figure}
We can see that in the case of growing $\SWT^V$ after $\SWT^U$, the labels belonging to explored vertices in $\SWT^U$ can not be used again, leading to some extra thinned vertices in $\SWT^V$. We claim that the number of additional ghosts is not too big. (Since we would like to get a bound on the effective size of active vertices in $SWT^V(t_n+u)$, we must delete the descendants of vertices that formed earlier collision events.)

\begin{claim} \label{claim:extrathin}
Consider the case of growing $\SWT^V$ after $\SWT^U(t_n)$ on the same graph $\NWn(\rho)$. Then the effective size of the active set in $\SWT^V$ for times $t=t_n+s$ is asymptotically the same as the size of the active set, that is, the statement of Corollary \ref{cor:effectivesize} remains valid for $\SWT^V$ as well.
\end{claim}
Note that it suffices to bound the proportion of ghosts, as the error terms caused by multiple active, or active-explored vertices are not increased by the presence of $\SWT^U$.
\begin{proof} We consider the computations in the proof of Lemma \ref{lemma:ghosts}, using \eqref{eq:ghostprob}. Recall that the proportion of ghosts depends simultaneously on the thinning probability of the $i\ith$ explored vertex as well as it being an ancestor of a uniform active vertex.

The arguments with the ancestral line (see \ref{p:ancestralline}) remain valid without any modification, we only have to examine the change in the thinning probability.

In case the $i\ith$ explored vertex is blue, its label is chosen uniformly, thus the probability that this label coincides with a previously chosen label equals $\left(\rN^U(t_n) + i-1\right)/n$. In case the $i\ith$ explored vertex in $\SWT^V$ is red, we can use the same idea as before: it has the label of an already explored vertex if and only if two intervals grow into each other with the $i\ith$ step. We now consider the union of the intervals in $\SWT^U$ and $\SWT^V$. Conditioned on the interval that grows, for any interval the probability that these two grow into each other is $2/n$. The number of intervals is at most the total number of blue explored vertices, $\rN_B^U(t_n)+ \rN_B^V(T_i)$. Hence the probability that the labeling fails at the $i\ith$ step of $\SWT^V$, if this is a red vertex is at most
\[ \sum_{k=1}^{\rN^U_B(t_n)+N_B^V(T_i)} \frac{2}{n} = \frac{2\left(\rN^U_B(t_n) + N_B^V(T_i)\right)}{n} \leq \frac{2\left( \rN^U(t_n) + i \right)}{n}. \]
Since the color of the i-th explored vertex is either blue or red, we get
\[ \Pv(\text{\ labeling fails in $\SWT^V$ at step i\ }) \le 2 \left( \rN^U(t_n) + i\right)/n. \]
For the probability of a uniformly chosen active vertex in $\SWT^V$ being a ghost, similarly to \eqref{eq:ghostprob}, we have
\be\label{eq::extra-skinny} \frac{\rA_G^V(t_n+s)}{\rA^V(t_n+s)} = \sum_{i=1}^{\rN^V(t_n+s)} \frac{D_i}{S_i} \!\cdot\! \frac{2(\rN^U(t_n) + i )}{n} =\sum_{i=1}^{\rN^V(t_n+s)} \frac{D_i}{S_i}\! \cdot\! \frac{2i}n + \sum_{i=1}^{\rN^V(t_n+s)} \frac{D_i}{S_i} \!\cdot\! \frac2n \rN^U(t_n).  \ee
By Lemma \ref{lemma:ghosts}, the first sum on the rhs tends to 0 as $n$ tends to $\infty$. for the second sum, let us recall the a.s. finite K in  Lemma \ref{lemma:si}, and we split the sum again. We use Corollary \ref{corollary:dead}, \ref{def:wn} and $t_n=\log n/(2\la)$ to get
\[  \frac2n \rN^U(t_n) \sum_{i=1}^{K} \frac{D_i}{S_i} \leq \frac2n \rN^U(t_n) \sum_{i=1}^{K} 1 = K \frac2{\la\sqrt{n}} W_U^{\scriptscriptstyle{(n)}} \toninf 0.\]
For the second part of the sum, by Lemma \ref{lemma:si} again,
\[ \Sigma_2 = \frac2n \rN^U(t_n) \sum_{i=K+1}^{\rN^V(t_n+s)} D_i/S_i = \frac2n \rN^U(t_n) \sum_{i=K+1}^{\rN^V(t_n+s)} \frac{D_i}{i\lambda(1+o(i^{-1/2+\ve}))} \]
Using $\E[D_i] \leq \lambda+2$, we bound the expected value of the sum $\Sigma_3:=\sum_{i=K+1}^{\rN^V(t_n+s)} \frac{D_i}{i\lambda(1+o(i^{-1/2 + \ve}))}$ with tower rule.
\[  \E\!\left[\Sigma_3\right] \!\le
 \E\!\left[(1+o(1)) \sum_{i=K+1}^{\rN^V(t_n+s)} \!\frac{\la+2}{i\lambda} \right]
\!\le \frac{\la+2}{\la} \E\Big[\log (\rN^V(t_n+s))\Big] (1+o(1))  \]
Since the logarithm is concave, we use  Jensen's inequality:
\be\label{eq::sigma3} \E[\Sigma_3] \leq
\frac{\la+2}{\la} \log \E\Big[\rN^V(t_n+s)\Big]  (1+o(1))
  \ee
  From Theorem \ref{thm::infinitesimal} it follows that $\E\big[\rN^V(t_n+s)\big]=(0,1) \exp\{Q\!\cdot\!(t_n+s)\} \mathbf 1$, where $\mathbf 1=(1,1)^T$ is a column vector.  Using the Jordan decomposition of the matrix $Q$ and exponentiating, elementary matrix analysis yields that the leading order is determined by the main eigenvalue $\la$ and hence $(1,0)\exp\{Q (t_n+s)\}\mathbf 1\le e^{\la (t_n+u)} C_1$ for some constant $C_1\ge 0$. Let us then use this bound with $t_n+s=\log n/(2\la) +s$ to give an upper bound on the rhs of \eqref{eq::sigma3}, and set $C_2:=2C_1(\la+2)$.
Then Markov's inequality yields
\[ \P\left( \Sigma_3 \geq C_2 (\log n)^2   \right) \leq 1/\log n \toninf 0.\]
Then on the w.h.p. event $\{\Sigma_3 \leq C_2 (\log n)^2  \}$ for $\Sigma_2 = \frac2n \rN^U(t_n) \Sigma_3$, by Corollary \ref{corollary:dead}, \eqref{def:wn} again,
\[  \Sigma_2 \leq \frac2n \rN^U(t_n) C_2(\log n)^2
\leq  C_2(\log n)^2  \e^{\la u}  \frac{W_V^{\scriptscriptstyle{(n)}} (1+o(1))}{\la\sqrt{n}} \toninf 0.  \]
\end{proof}

\subsection{The Poisson point process of collisions}
Recall that we say that a collision event happens at time $t_n+s$ when an active vertex in $\SWT^U(t_n)$ becomes explored in $\SWT^V$ at time $t_n+s$.
First we show that for each pair of colours, with respect to the parameter $s$ in $t_n+s$, the set of points $(P_i)_{i\in \N}$ form a non-homogeneous Poisson point process (PPP) on $\R$, and that these PPP-s are asymptotically independent. 
We consider the intensity measure $\mu(\mathrm d t)$, $t \in \R$ as the derivative of the mean function $M(t)$ (expected value of points up till time $t$).
To determine the intensity measure of the collision process, we will consider the four collision point processes for each possible pair of colours. None of the PPPs is empty: since the labels of blue vertices are chosen uniformly, they can meet any color, and considering the growing set of intervals, we see that red can meet red as well. \\
Let us introduce the notation for $q,r \in \{R,B\}$, $s\in \R$
\be \label{def::connection} \ba \cC_{q,r}(s)&:=\cN^V_{q}(t_n+s) \cap \cA^U_{r}(t_n), \quad \rC_{q,r}(s):=|\cC_{q,r}(s)|,\\
\cC(s)&:=\bigcup_{q,r \in \{R,B\}} \cC_{q,r}(s), \quad \rC(s):=|\cC(s)|\ea \ee
(Note that e.g.\ $\cC_{R,B}(s)$ denotes the set of \textit{red} explored labels in $SWT^V$ that are \textit{blue} active in $SWT^U$.) The corresponding intensity measures are denoted by $\mu_{q,r}(s)$ and total intensity measure by $\mu(s)$.

In Corollary \ref{cor:effectivesize} we showed that the effective size of the active set at times $t_n+s$ for $s\in \R$ is asymptotically the same as the number of active individuals in the CMBP, so we can use the asymptotics of $(A_R(t), A_B(t))$ from Theorem \ref{theorem:actives}. We shall handle the cases when blue vertices are involved in a similar fashion. For the coming sections, for $s\in \R$ and and event $E$ we  define the notation $\Pv_s(E):=\Pv(E| \rA^U_{B}(t_n), \rA^U{R}(t_n), \rN^V_B(t_n+s), \rN^V_R(t_n+s))$.
\subsubsection*{\textit{Blue-blue collision}}
By the definition of the set $\cC_{B,B}(s)$, we can use the following indicator representation:
\be\label{eq::poi-BB} \rC_{B,B}(s) = \!\!\!\!\!\!\!\!\sum_{x \in \cN^V_{B}(t_n+s)} \!\!\!\!\!\!\!\!\ind \{x \in \cA^U_{B}(t_n)\}. \ee
Recall that the labels in $\cA^U_B(t_n)$ are chosen uniformly in $[n]$, hence for any label $x\in [n]$,
\be\label{eq::label-calc} \P_s(x \in \cA^U_{B}(t_n))=1-(1- 1/n)^{\rA^U_B(t_n)} = \rA^U_B(t_n)/n (1+o(1)).\ee
Further, since $\rA^U_B(t_n)=O(\sqrt{n})$, elementary calculation using inclusion-exclusion formula shows that the indicators in \eqref{eq::poi-BB} are weakly dependent in the sense that $\Pv_s(y \in \cA^U_{B}(t_n) | x \in \cA^U_{B}(t_n))=\rA^U_B(t_n)/n (1+o(1))$.
Further, the events
$\{x \in \cN^V_{B}(t_n+s) \}, \{x \in \cA^U_{B}(t_n)\}$ are independent, the usual Poisson approximation for weakly dependent random variables (e.g.\ via factorial moments, as in \cite[Theorems 2.4, 2.5]{Hofs16}) yields that $\rC_{B,B}(s)$ converges to a Poisson variable for each $s\in \R$, and by Wald's equation,
\[ \E_s[\rC_{B,B}(s)] = \rN^V_{B}(t_n+u) \cdot \rA^U_B(t_n)/n (1+o(1)). \]
By noting that $\ind \{x \in \cN^V_{B}(t_n+s_1)\} \le \ind \{x \in \cN^V_{B}(t_n+s_2)\}$ when $s_1\le s_2$, we also get that $\rC_{B,B}(s_1)\le \rC_{B,B}(s_2)$ and hence by standard methods one can show that the process $(\rC_{B,B}(s))_{s\in \R}$ converges to a Poisson point process in $\R$.
\subsubsection*{\textit{Red-blue collision}}
Using the same argument as for the previous case,
\[
\rC_{R,B}(s) = \!\!\!\!\!\!\sum_{x \in \cN^V_{R}(t_n+s)} \!\!\!\!\!\!\ind\{x \in \cA^U_{B}(t_n)\}, \quad \E_s[\rC_{R,B}(s)] =  \rN^V_{R}(t_n+s) \cdot \rA^U_{B}(t_n)(1+o(1)). \]
Further, for $y\in \cN^V_{R}(t_n+s)$ and $x \in \cN^V_{R}(t_n+s)$, the indicators $\ind\{z \in \cA^U_{B}(t_n)\}$ for $z=x,y$ are only weakly dependent, we can also conclude that the process $(\rC_{R,B}(s))_{s\in \R}$ is asymptotically independent of $(\rC_{B,B}(s))_{s \in \R}.$

\subsubsection*{\textit{Blue-red collision}}
%
This time we approach the number of collision events through indicators of labels in the explored set, i.e.,
\be\label{eq::poi-RB} \cC_{B,R}(s) = \cN^V_{B}(t_n+s) \cap \cA^U_{R}(t_n), \quad
\rC_{B,R}(u) = \!\!\!\!\!\!\sum_{x \in \cA^U_{R}(t_n)} \!\!\!\!\!\!\ind\{x \in \cN^V_{B}(t_n+s)\}. \ee
The blue labels in $\cN^V_{B}(t_n+s)$ are chosen uniformly in $[n]$, hence, a similar calculation  as in \eqref{eq::label-calc}, and the independence of the labeling in $\cN^V_{B}(t_n+s)$ and in $\cA^U_{R}(t_n)$ yields again
\[ \E_s[\rC_{R,B}(s)] =\rA^U_{R}(t_n)\cdot \rN^V_{B}(t_n+s)/n(1+o(1)). \]
The independence of the three processes $\big(\rC_{R,B}(s), \rC_{R,B}(s), (\rC_{B,R}(s) \big)_{s\in \R}$ can be seen by `reversing' the indicators in \eqref{eq::poi-BB} to get a similar form as the one in \eqref{eq::poi-RB}, and noting that the variables included are again only weakly dependent.

Using the asymptotic results for $\mathbf A(t), \mathbf N(t)$ (from Theorem \ref{theorem:actives} and Corollary \ref{corollary:dead}) and the definition of $W^{(n)}$ in \eqref{def:wn}, then differentiating with respect to $s$ yields
\be\label{eq::intensity-measures} \ba
\mu_{B,B}(\mathrm ds) &=\pi_B^2 \Wvn \Wun \e^{\lambda s} \mathrm ds (1+o(1)), \;\; \\
\mu_{R,B}(\mathrm ds) &=\pi_B \pi_R \Wun \Wvn \e^{\lambda s} \mathrm ds (1+o(1)), \;\; \\
\mu_{B,R}(\mathrm ds) &= \pi_B \pi_R \Wvn \Wun \e^{\lambda s} \mathrm ds (1+o(1)), \ea \ee
where the $1+o(1)$ factor only depends on $n$ and comes from the error of the approximation $N_q(t)= e^{\la t}\la^{-1} W^{\scriptscriptstyle{(n)}} \pi_q (1+o(1)).$ It is not hard to show (e.g. using Borel-Cantelli lemma) that these PPP-s have only finitely many points on $(-\infty, 0)$, hence, indexing the points by $i\in \N$ is doable. 
\subsubsection*{\textit{Red-red collision}}
The red-red collision events have to be treated slightly differently, since here the geometry of the cycle plays a more important role. Recall that we stopped the evolution of $\SWT^U$ at time $t_n$. Hence, we consider $\SWT^U(t_n)$ as a fixed set of intervals, $\{I_k, k=1,...,\rN^U_B(t_n)\}$ (some of them might have already possibly merged by time $t_n$), while the set of intervals $\{J_i(s), i=1,...,\rN^V_B(t_n+s)\}$ in $\SWT^V(t_n+s)$ is growing (and possibly merging) with $s$. A collision happens when one of the intervals $J_i(s)$ grows into one of the intervals $I_k$.

Note that here we face a new technical issue.
When two intervals $J_i(s)$ and $I_k$ collide at time $s$, in principle we should stop the evolution of $J_i(s)$, that is, for all $s'>s$ we should have $J_i(s')\equiv J_i(s)$. But, since this would cause computational difficulties later, (since then we have to condition on all the earlier collisions to be able to calculate the intensity of the next one). Hence, it is easier to do the following approximation on the number of red-red collisions: we let $J_i(s)$ grow further and it might collide with more vertices inside $I_k$. The error terms caused by such events is negligible, since such events have been already treated when we investigated the `extra' thinning of $\SWT^V$ imposed by $\SWT^U(t_n)$, that had a negligible contribution in the sense of Claim \ref{claim:extrathin}.

We decompose the number of collisions as a sum of indicators. Recall the notation from paragraph \ref{p:redsplit}: $c_k, l_k, r_k$ stands for the location of the blue vertex, the number of red explored to the left and right from $c_k$ in the interval $I_k$. We adapt the same notation for the intervals $J_i$ by adding a superscript $V$ to the above quantities. We claim that for any pair $I_k$ and $J_i(s)$, the probability that they had already made a red-red collision by time $t_n+s$ is
\be\ba \label{eq:collision-ind} \{ J_i(s) \cap I_k \neq \emptyset \text{ via red-red collision} \} &= \cup_{m=1}^{l_i(s)} \ind\{c^V_i = c_k - l_k-1-m \} \\
& \ \ \bigcup \cup_{m=1}^{r_i(s)} \ind\{c^V_i = c_k + r_k+1+m \}. \ea\ee
To see this, we condition on $c_k, r_k, l_k$. Note that there is two active red vertices at the boundary of $I_k$: one on the left, at location $c_k-l_k-1$ and one on the right, at location $c_k+r_k +1$. When the one at location $c_k-l_k-1$ coincides with any of the $r^V_i(s)$ explored right-type red vertices of $J_i(s)$, a (potentially past) collision has happened. Each of these events uniquely determines the position of $c^V_i$, the location of the blue vertex of $J_i(s)$, forming an interval of length $r^V_i(s)$ for potential locations. We remark that this left active vertex at location $c_k-l_k-1$ can \emph{not} coincide with a left-type red vertex in $J_i(s)$, since that would mean that either $c^V_i$ also coincides with $I_k$ -- causing an event that we already counted in Claim \ref{claim:extrathin}, or, $c^V_i$ is to the left of the whole interval $I_k$, and $l^V_i(s)$ is already so large that it swallowed the whole interval $I_k$, and this case also have been already treated in Claim \ref{claim:extrathin}. (Since in this case, the part of $J_i(s)$ that the overlaps with $I_k$ has been already thinned.)
Similarly, there are $l^V_i(s)$ many positions for $c^V_i$ on the left of $c_k$ that cause a valid collision event. Using that $c^V_i$ is chosen uniformly over $[n]$ -- since it is a blue vertex -- yields the right hand side of \eqref{eq:collision-ind}.
Note that these indicators are mutually exclusive, and hence
\be\label{eq:collision-prob} \P_s(J_i(s) \cap I_k \neq \varnothing \text{\ via red-red collision} )=(l^V_i(s)+r^V_i(s))/n. \ee
The number of red-red collisions can be obtained by summing up the right hand side of \eqref{eq:collision-ind} for all intervals $I_k$ and $J_i(s)$. This yields a collection of indicators that have a weak dependence structure: each indicators only depends on at most $\max_{i}{|J_i(s)|}$ many other indicators. Since $t=O(\log n)$, and the length of each interval is determined by consecutive exponential variables (i.e. the time to reach the $k$-th red vertex on the right of a blue vertex has a $\mathrm{Gamma}(k,1)$ distribution), it is elementary to show that $\max_{i}{|J_i(s)|}=o((\log n)^2)$ asymptotically almost surely. As a result, a Poisson approximation can be carried out by using e.g.\ the factorial moments of these indicators, in a similar fashion that the one in \cite[Proposition 7.12]{Hofs16}, using \cite[Theorems 2.4, 2.5]{Hofs16}.  We leave the details to the reader. We get that $\rC_{R,R}(s)$ converges to a PPP with intensity measure given by
\[ \E_s[\rC_{R,R}(s)] = \sum_{k=1}^{\rN^U_B(t_n)} \sum_{i=1}^{\rN^V_B(t_n+s)} (l^V_i(s)+r^V_i(s))/n. \]
Note that the inner sum is simply $\rN^V_R(t_n+s)$, while  $\rN^U_B(t_n)=\rA^U_R(t_n)/2$, since there are two red actives in each interval $I_k$. Then,
\[ \E_s[\rC_{R,R}(s)] = \rA^U_R(t_n) \cdot \rN^V_R(t_n+s)/(2n). \]
Further, similar arguments as before show that this PPP is asymptotically independent of $\rC_{B,B}(s), \rC_{B,R}(s), \rC_{R,B}(s)$, and we also obtain \begin{equation} \label{eq:murr}
\mu_{R,R}(\mathrm ds) = \frac12 \pi_R^2 \Wun \Wvn \e^{\lambda s} \mathrm d s (1+o(1)).
\end{equation}
We emphasize that to get \eqref{eq::intensity-measures} and \eqref{eq:murr} we assumed that the number of actual intervals in $\SWT^U(t_n)$ and $\SWT^V(t_n+s)$ is $N_B^U(t_n)$ and $N_B^V(t_n+s)$, respectively. This is not entirely true due to the fact that intervals within $\SWT^U$ or within $\SWT^V$ might have merged already earlier. However, in this case some of the included vertices are ghosts:  Corollary \ref{cor:effectivesize} showes that the effective size of the active set at times $t_n+s$ for $s\in \R$ is asymptotically the same as the number of active individuals in the CMBP, which, by the fact that every interval has precisely two active red vertices, implies that also the number of \emph{disjoint} active intervals is asymptotically the same as $N_B^U(t_n)$ and $N_B^V(t_n+s)$, respectively in the two processes. We have arrived to the following theorem:
\begin{theorem}[Total intensity measure of the collision PPP] \label{theorem:mutotal}
The total intensity measure of the collision Poisson point process is
\be\label{eq::mu}\mu(\mathrm ds) = \Wun \Wvn \e^{\lambda u}\! \left( \pi_B^2 + 2\pi_B\pi_R + \pi_R^2/2\right) = \Wun \Wvn \e^{\lambda s}\! \left(1 -  \pi_R^2/2\right) \!\mathrm ds (1+o(1)), \ee
where the factor $(1+o(1))$ only depends on $n$.
\end{theorem}
\begin{proof}
We have calculated the individual intensity measures of the four Poisson point processes in \eqref{eq::intensity-measures} and \eqref{eq:murr}. These processes  are asymptotically independent, since the total collection of \emph{all} the indicator variables included in their construction are only weakly dependent. One way then is to show that the joint factorial moments converge, in a similar manner than the one in \cite[Proposition 7.12]{Hofs16}. Or, one might also consider the point process of the total intensities immediately together to get \eqref{eq::mu}. It is important to observe that the depletion of the number of active individuals due to earlier collisions does not influence the asymptotic intensities, since the first is of constant order on the time scale $t=t_n+s, s\in \R$, while the latter is of order $\sqrt{n}$, the size of the active/explored labels in each of the colors.

\end{proof}
\subsection{Proof of Theorem \ref{thm:main}} \label{ss:mainproof}
\begin{figure}[ht]
\centering
\includegraphics[keepaspectratio, width=9cm]{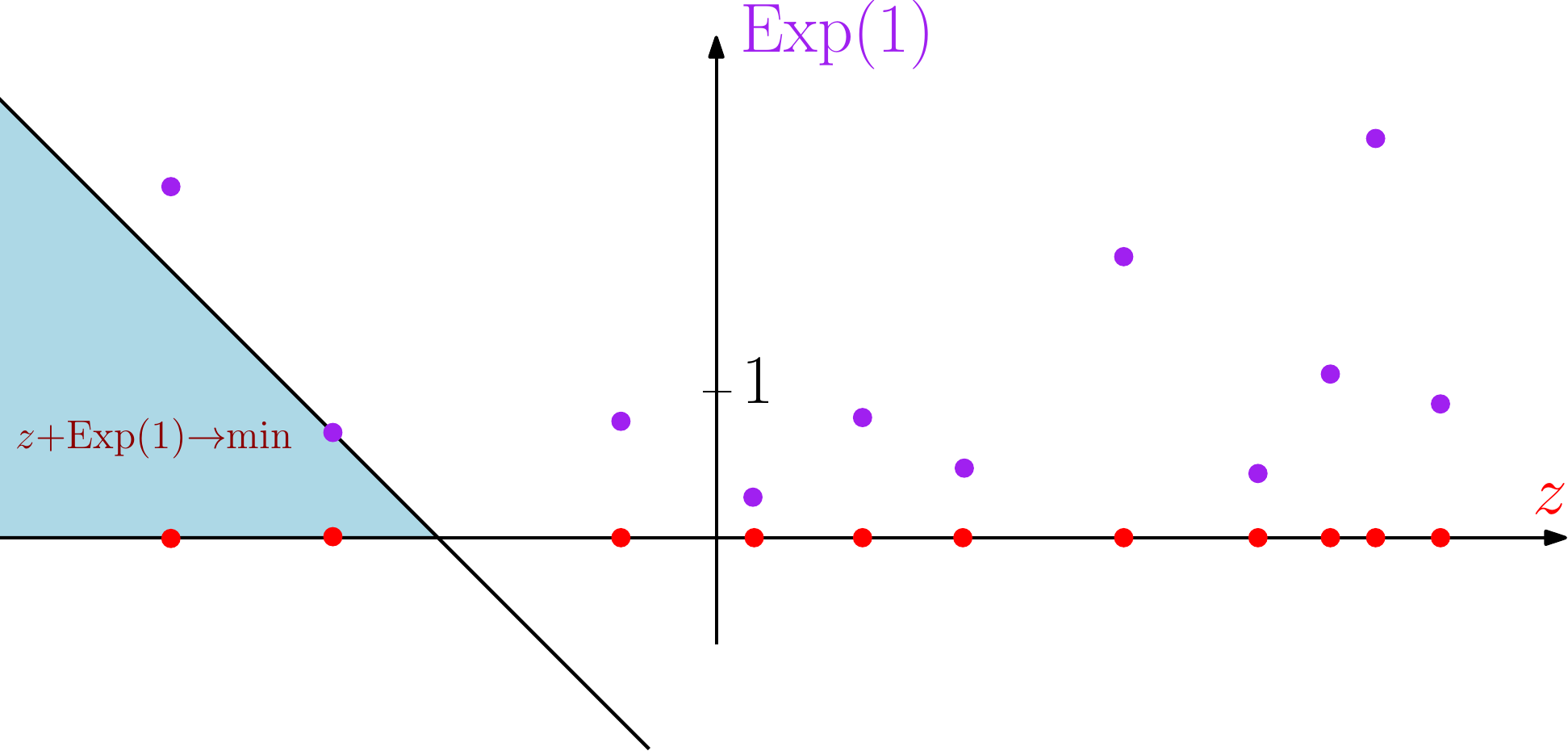}
\caption{An illustration of the two dimensional PPP with points $(P_i, E_i)_{i\in \N}$, with the point minimizing the sum $P_i+E_i$ indicated.}\label{fig::ppp}
\end{figure}
It is well known \cite{Res86Poi} that if $(E_i)_{i\in N}$ is a collection of i.i.d. random variables with distribution function $F_E(y)$, and the points $(P_i)_{i\in \N}$ form a one-dimensional Poisson point process with intensity measure $\mu(\mathrm ds)$ on $\R$, then the points $(P_i, E_i)_{i\in \N}$ form a two-dimensional non-homogeneous Poisson point process on $\R \times \R$, with intensity measure $\mu(\mathrm ds) \times F_E(\mathrm dy)$.

In our case, to get the shortest path between $U$ and $V$, recall from Definition \ref{def:collision-connection} that we have to minimize the sum of time of the collision and the remaining lifetimes over the collision events. Mathematically, we want to minimize the quantity $P_i+E_i$ over all points $(P_i,E_i)$ of the two dimensional PPP with intensity measure $\nu(\mathrm ds \times \mathrm dy):=\mu(\mathrm ds) \times \e^{y} \mathrm dy$, since the remaining lifetimes are i.i.d.\ exponential random variables.

Note that event $\{\min_i P_i+E_i \ge z\}$ is equivalent to the event that there is no point in the infinite triangle $\Delta(z) = \{(x,y): y>0, x+y<z\}$ in this two-dimensional PPP.
%
%
We calculate
\[ \ba \nu(\Delta(z)) = \int_{-\infty}^{z} \int_{0}^{z-s} \mu(\mathrm ds) \cdot \e^{-y} \rd y
= \Wun \Wvn \left(1-\frac12\pi_R^2\right) \frac{1}{\lambda(\lambda+1)} \e^{\lambda z}(1+o(1)). \ea \]
For short, we introduce
\begin{equation} \label{eq:cw}
\cWn := \Wun \Wvn \left(1-\frac12\pi_R^2\right) \frac{1}{\lambda(\lambda+1)}(1+o(1)),
\end{equation}
Then we can reformulate
\be \label{nudeltaz} \nu(\Delta(z)) = \cWn \e^{\lambda z} \ee
Let us turn our attention back to $\CP_n(U,V)$, the shortest weight path between $U$ and $V$. By the previous argument, we conclude
\be \label{eq:poi0} \P(\CP_n(U,V) \geq z+2t_n|\cWn)
= \P(\Poi(\nu(\Delta(z))) = 0|\cWn) = \exp\{-\nu(\Delta(z))\} \ee
Rearranging the left hand side and substituting the computed value of $\nu(\Delta(z))$, we get
\[ \P(2t_n-\CP_n(U,V) \leq -z| \cWn) = \exp\{- \cWn \e^{\lambda z} \} \]
We substitute $t_n=\log n/(2\la)$, and set $z:=-x/\la$ to get
\[  \P\left( \log n -\lambda \CP_n(U,V)< x | \cWn \right)
= \exp\{ - \exp \{ -(x - \log \cWn)   \} \}  \]
We recognize on the right hand side the cumulative distribution function of a shifted Gumbel random variable, which implies
\[ (\log n-\lambda \CP_N(U,V) ) | \cWn {\ \buildrel d \over =\ }  \Lambda + \log \cWn,\]
Rearranging and substituting $\cWn$ from \eqref{eq:cw},
and using that the martingales $(\Wun,\Wvn) \toas (W_U, W_V)$,
\[ \CP_n(U,V) - \frac1\lambda \log n \toindis = - \frac1\lambda \Lambda - \frac1\lambda \log(W_U W_V) - \frac1\lambda \log \left(1-\frac12\pi_R^2\right) + \frac1\lambda \log \left( \lambda (\lambda+1) \right). \]
This finishes the proof of Theorem \ref{thm:main}. 
\section{Epidemic curve} \label{s:ec}
Recall the definition of the epidemic curve function from Section \ref{ss:mainresults}. The discussion of the epidemic curve consists of three parts: first, we find the correct function $f$  by computing the first moment of $\rI_n(t,U)$. Then we prove the convergence in probability by bounding the second moment. Finally, we give a characterization of $M_{W_V}$, the moment generating function of the random variable $W_V$, that determines the epidemic curve function $f$.

\subsection{First moment}
First, we condition on the value $\Wun$ from the martingale approximation of the branching process of the uniformly chosen vertex $U$. Then we can express the fraction of infected individuals as a sum of indicators, and calculate its conditional expectation:
\[ \E \left[ \rI_n(t,U)\left| \Wun \right.\right]
= \frac1n \sum_{w \in [n]} \P\left( w \text{ is infected by time t} \left| \Wun \right.\right). \]
 Note that the rhs equals the probability of a \emph{uniformly chosen} vertex, which we shall denote by $V$,  being infected. Also note that a vertex is infected if and only if its distance from $U$ is shorter than the time passed, hence
\be\label{eq::ecin} \E \left[ \rI_n(t,U)\left| \Wun \right.\right]
=  \P \left( \CP_n(U,V) \leq t \left| \Wun \right. \right). \ee
Now we can further condition on $\Wvn$ and use the distribution of $\CP_n(U,V)$ conditioned on $\Wun, \Wvn$. Recall from  Subsection \ref{ss:mainproof}, that
\begin{equation} \ba \label{ecineq} &\P\left.\left( \CP_n(U,V) \geq z+2t_n \right| \Wun, \Wvn \right)
= \exp\left\{ -\Wun \Wvn \left(1-\tfrac12 \pi_R^2\right)\tfrac1{\lambda(\lambda+1)}\e^{\lambda z}\right\}. \ea \end{equation}
Let us set here $z=t-\frac1\lambda \log \Wun$ and rearrange, yielding 
\[ \ba &\P\left( \CP_n(U,V) \leq t-\tfrac1\lambda \log \Wun + \tfrac1\lambda \log n \left| \Wun, \Wvn \right.\right) \\
&= 1- \exp\left\{ - \Wvn \left(1-\tfrac12 \pi_R^2\right)\tfrac1{\lambda(\lambda+1)}\e^{\lambda t}\right\} \ea \]
Then, from \eqref{eq::ecin} and \eqref{eq::ecin2} we get
\be\label{eq::ecin2} \ba \E \left[ \rI_n(t-\tfrac1\lambda \log \Wun + \tfrac1\lambda \log n,U)| W_U^{(n)}\right]
= 1-\E\Big[\exp\left\{ - \Wvn \left(1-\tfrac12 \pi_R^2\right)\tfrac1{\lambda(\lambda+1)}\e^{\lambda t}\right\}\Big]. \ea \ee
We recognize that the second term on the right hand side is the moment generating function of $\Wvn$, $\rM_{\Wvn}(x)$, at $x(t) = - \left(1-\frac12 \pi_R^2\right)\frac1{\lambda(\lambda+1)}\e^{\lambda t}$. \\
Changing variables yields
\[ \E[\rI_n(t + \tfrac1\lambda \log n,U)| \Wun] = 1 - M_{\Wvn}\left( x(t + \tfrac1\lambda \log \Wun)\right). \]
Note that $\Wvn$ converges to $W_V$ almost surely, which implies their moment generating functions converge in probability. This function is exactly the one given in Theorem \ref{thm:ec}. 

\subsection{Second moment}
The first moment above showed that the expected value of $\rI_n(t,U)$ indeed converges in probability to the defined $f$ function at the given point. We prove Theorem \ref{thm:ec} by showing that the variance of $\rI_n(t,U)$ converges to 0, then Chebyshev's inequality yields that $\rI_n(t,U)$ converges to its expectation in probability.

Denote by $\ind_i = \ind\{i \text{ is infected by time $t$}\}$.
Let us calculate 
\[ \Var\Big(\frac1n \sum_{i \in [n]} \ind_i| \Wun\Big) \\
= \frac1{n^2} \sum_{i \in [n]} \Var(\ind_i|\Wun) + \frac2{n^2}\sum_{i<j \in [n]} \Cov[\ind_i,\ind_j| \Wun] \]
Since $\ind_i$ is an indicator, $\Var(\ind_i|\Wun) \leq 1$, hence the first term on the rhs is at most $\frac1n$. As for the second term,
\[ \ba &\Cov[\ind_i,\ind_j| \Wun] = \E[\ind_i\ind_j | \Wun] - \E[\ind_i | \Wun] \E[\ind_j |\Wun] \\
&= \P(i \text{ and } j \text{ are both are infected} | \Wun) - \P(i \text{ is infected}|\Wun)\P(j \text{ is infected}|\Wun) \ea \]
Imagine now three exploration processes on $\NWn$, one from $U$, one from $i$ and one from $j$. It is not hard to see that the three exploration processes from these three vertices can be approximated by three \emph{independent} branching processes. This implies that the covariance can be bounded by the error of coupling between the graph and the branching processes, as well as the thinning inside one tree and between the trees: these all have error terms of order at most $1/\log n$. It is not hard to see that the coupling can be extended to three $\SWT$'s (instead of two, as before), and the error terms increase only by constant multiples. The connection processes between $\SWT^U$ and the other two are related only through the intersection of $\SWT^i$ and $\SWT^j$, which is again at most of order $1/\log n$. As a result, then $\P(i \text{ and } j \text{ are both are infected} | \Wun)-\P(i \text{ is infected}|\Wun)\P(j \text{ is infected}|\Wun)=O(1/\log n)$.

This coupling works if $i$ and $j$ are fairly apart, say $(i-j) \, \mathrm{mod}\ n \,> (\log n)^{1+\ve}$ for any $\ve>0$ (this is w.h.p. longer than the length of the longest red interval). The number of "bad pairs", which are closer than this is $n (\log n)^{1+\ve}/2$, compared to the number of all pairs $\binom n2$, the fraction goes to 0. Even for these, the covariance is bounded by 1. Then the sum divided by $n^2$ goes to 0.

With that, we have bounded the variance by a term that goes to 0, which finishes the proof.

\subsection{Characterization of the epidemic curve function} In this section, we prove Proposition \ref{funceq}.
Recall that adding a superscript $(B)$ or $(R)$ indicates a branching process described in Section \ref{ss:bp} that is started from a blue or red type vertex, respectively. 
We start with the recursive formula for the martingale limit random variables from \cite{AthNey72BP}:
\[\WB \eqindis \sum_{i=1}^{D^{(B)}_R} \e^{-\lambda X_i} \WR_i + \sum_{j=1}^{D^{(B)}_B} \e^{-\lambda X_j} \WB_j, \]
where $\WR_i$ are independent copies of $\WR = \lim_{t\to\infty} \e^{-\lambda t} Z^{\scriptscriptstyle{(R)}}(t)$, and $\WB_j$ are independent copies of $\WB \eqindis W_V$, and $X_i, X_j$ are i.i.d. $\mathrm{Exp}(1)$. Denote the moment generating functions of $\WB$ and $\WR$ by $M_{\WB}$, $M_{\WR}$ respectively.
Recall that a blue individual has two red and $\Poi(\rho)$ many blue children. Hence
\be\label{eexp} \textstyle \E\left[\e^{\vartheta \WB}\right] = \left(\E\left[\exp\{\vartheta\e^{-\lambda X_i}\WR\}\right]\right)^2 \cdot \E \left[\exp\left\{\vartheta \sum_{j=1}^{D^{(B)}_B} \e^{-\lambda X_j} W^{(B)}_j\right\}\right] \ee
We use law of total expectation with respect to $X_i$ to compute
\[ \ba J^{(R)} &:= \E\left[\exp\{\vartheta\e^{-\lambda X_i}\WR\}\right] = \int_0^\infty \E\left[\exp\{\vartheta\e^{-\lambda x}\WR\}\right] \e^{-x} \rd x \\
&= \int_0^\infty M_{\WR} (\vartheta\e^{-\lambda x}) \e^{-x} \rd x \ea \]
Let $J^{(B)}$ defined similarly, with $M_{\WR}$ replaced by $M_{\WB}$. 
Then, the second factor in \eqref{eexp} can be treated by conditioning on $D^{(B)}_{B}$ and using independence:
\[  \textstyle \prod_{j=1}^{D^{(B)}_B} \E\Big[\exp\{\vartheta\e^{-\lambda X_j} \WB_j\} \Big]
:= \prod_{j=1}^{D^{(B)}_B} J^{(B)}
= \left(J^{(B)}\right)^{D^{(B)}_B}.  \]
Taking expectation w.r.t. $D^{(B)}_B\eqindis \Poi(\rho)$ yields that
\[ \E \left[ \exp\Big\{\vartheta \textstyle \sum_{j=1}^{D^{(B)}_B} \e^{-\lambda X_j} W^{(B)}_j\Big\} \right]
= \exp \left\{\rho \big(J^{(B)} -1\big)\right\} \]
We can rewrite the factor in exponent as
\[ \ba J^{(B)} -1 = \int_0^\infty M_{\WR} (\vartheta\e^{-\lambda x}) \e^{-x} \rd x - 1
= \int_0^\infty \left(M_{\WR} (\vartheta\e^{-\lambda x}) -1\right) \e^{-x} \rd x, \ea \]
then the moment generating function in \eqref{eexp} becomes
\[  M_{\WB} (\vartheta)
= \left( \int_0^\infty M_{\WR} (\vartheta\e^{-\lambda x}) \e^{-x} \rd x \right)^2 \,\cdot\, \exp \!\left\{\rho \cdot \int_0^\infty \left(M_{\WR} (\vartheta\e^{-\lambda x}) -1\right) \e^{-x} \rd x \right\}. \]
Similarly for $M_\WR$, using that $D^{(R)}_R = 1$ and $D^{(R)}_B \eqindis \Poi(\rho)$,
\[ M_{\WR} (\vartheta)
= \int_0^\infty M_{\WR} (\vartheta\e^{-\lambda x}) \e^{-x} \rd x \,\cdot\, \exp\! \left\{\rho \cdot \int_0^\infty \left(M_{\WR} (\vartheta\e^{-\lambda x}) -1\right) \e^{-x} \rd x \right\} \]
We have just showed that the moment generating functions satisfy the system of equations given in Proposition \ref{funceq}, and by \cite{AthNey72BP}, there exist proper moment generating functions satisfying these functional equations.

\section{Central limit theorem for the hopcount} \label{s:clt}
In this section we prove Theorem \ref{thm:clt} that states that the hopcount $\rH_n(U,V)$, the number of edges along the shortest weight path between two vertices  $U$ and $V$ chosen uniformly at random, satisfies a central limit theorem with mean and variance both $\frac{\la+1}{\la} \log n$.

For this, we consider the shortest weight path between $U$ and $V$ in two parts: the path from $U$ within $\SWT^U(t_n)$ and from $V$ within $\SWT^V$, to the vertex where the connection happens. We denote the vertex where the connection happens by $Y$. These paths are disjoint with the exception of $Y$, hence if suffices to determine their lengths, i.e. the graph distance of $Y$ from $U$ and $Y$ from $V$. Denote by $G^{(U)}(Y)$ the generation of $Y$ in $\SWT^U$, similarly for $V$. Then the required steps from the root $U$ to $Y$ is exactly $G^{(U)}(Y)$.

\begin{claim}
The choice of $Y$ is asymptotically independent in the two $\SWT$'s.
\end{claim}
\begin{proof}
Conditioned on $Y$ being the connecting vertex, it is uniformly chosen over the active set of $\SWT^V$. That determines its label, and it determines which particle is chosen in $\SWT^U$ \emph{through the label}. Since the labeling is independent of the structure of the family tree, aside from the thinning, the choice of $Y$ in $\SWT^U$ is independent from its choice in $\SWT^V$. We have already bounded the fraction of ghost particles (those who have one of their ancestors thinned) by a term that goes to 0 in Lemma \ref{lemma:ghosts}, hence asymptotic independence holds.
\end{proof}
With these notations, $\rH_n(U,V) = G^{(U)}(Y) + G^{(V)}(Y)$, and the two terms are independent. We reformulate the theorem using these terms:
\be\label{eq::uv-decomp1}
\frac{\rH_n(U,V) - \frac{\la+1}{\la} \log n}{\sqrt{\frac{\la +1}{\la} \log n}}
= \frac{G^{(U)}(Y) - \frac{\la+1}{2\la} \log n}{\sqrt{\frac{\la +1}{\la} \log n}} + \frac{G^{(V)}(Y) - \frac{\la+1}{2\la} \log n}{\sqrt{\frac{\la +1}{\la} \log n}}
 \ee
Considering that the terms are independent, it suffices to show that both
terms on the right hand side have normal distribution with mean 0 and variance $\tfrac12$. We show that both terms on the rhs of \eqref{eq::uv-decomp1}, multiplied by $\sqrt{2}$, have standard normal distribution. Due to the method we established the connection between $\SWT^U$ and $\SWT^V$, the two terms need to be treated somewhat differently. 

\subsection{Generation of the connecting vertex in $\SWT^V$}

Recall that we established the connection between $\SWT^U$ and $\SWT^V$ in the following way: we grew $\SWT^U$ until time $t_n$, then we freeze its evolution. Then, we grow $\SWT^V$, and every time a label is assigned to a splitting particle, we check if this label belongs to the active set of $\SWT^U$. As a result, the connecting vertex $Y$ is a particle at some splitting time $T_k$, and hence chosen uniformly over the active vertices. This implies that we can use for its generation the indicator decomposition of the ancestral line described in Section \ref{p:ancestralline}, for $Y$'s generation as $G_k = \sum_{i=1}^{k} \ind_i $, where conditioned on the offspring variables $D_i$, the indicators are independent and have success probability $\P(\ind_i=1)=\frac{D_i}{S_i}$.

In our case the number of splits is a random variable. Recall from Section \ref{ss:mainproof} that the connection time minus $t_n$ forms a tight random variable, (see e.g. \eqref{eq:poi0}), hence till the connection there are  $\rN(t_n + Z)$ many explored vertices for some random $Z \in \R$. By Corollary \ref{corollary:dead}, $\rN(t_n+u)=C \sqrt{n}$ for some bounded random variable $C$ (that might depend on $n$, but is tight).
Denote by
\be\label{eq::b123} \ba
B_1 &= \frac{\sum_{i=1}^{C\sqrt{n}}\ind_i - \sum_{i=1}^{C\sqrt{n}}\frac{D_i}{S_i}} {\sqrt{\sum_{i=1}^{C\sqrt{n}} \frac{D_i}{S_i}\left(1-\frac{D_i}{S_i}\right)}}, \quad
B_2 = \frac{\sqrt{ \sum_{i=1}^{C\sqrt{n}} \frac{D_i}{S_i}\left(1-\frac{D_i}{S_i}\right)}}{\sqrt{\frac{\la +1}{2\la} \log n}}, \\
B_3 &= \frac{\sum_{i=1}^{C\sqrt{n}} \frac{D_i}{S_i} - \frac{\la+1}{2\la}\log n}{\sqrt{\frac{\la +1}{2\la} \log n}}
\ea \ee
Then
\be\label{eq::gyv}
\frac{G^{(V)}(Y) - \frac{\la+1}{2\la} \log n}{\sqrt{\frac{\la +1}{2\la} \log n}}=
\frac{\sum_{i=1}^{C\sqrt{n}}\ind_i - \frac{\la+1}{2\la} \log n}{\sqrt{\frac{\la +1}{2\la} \log n}}= B_1 \!\cdot\! B_2 + B_3\ee
Our aim is to show that Lindeberg's CLT is applicable for $B_1$, $B_2$ converges to 1, and $B_3$ converges to 0.

\subsubsection{Term $B_1$}\label{ss::b1}
For this sum of (conditionally independent) indicators, Lindeberg's condition is trivially satisfied if the total variance tends to infinity. To give a lower bound,
\be\label{eq::variance1} \sum_{i=1}^{C\sqrt{n}} D_i/S_i\left(1-D_i/S_i\right) = \sum_{i=1}^{C\sqrt{n}}D_i/S_i - \sum_{i=1}^{C\sqrt{n}}D_i^2/S_i^2. \ee
Recall Lemma  \ref{lemma:si}, and split the sum according to the random variable $K$. Each vertex has at least one red child, hence $D_i\geq1$. Then
\[ \ba \sum_{i=1}^{C\sqrt{n}}D_i/S_i
 \geq \sum_{i=K+1}^{C\sqrt{n}}\frac1{i\lambda(1+o(i^{-1/2+\ve}))}. \ea \]
where $K$ is a.s.\ finite.  The $i\ith$ term on the rhs is at least $1/(2i)$, and thus the rhs tends to infinity, and is at least $\log n/(2\la)$.
For the second term in \eqref{eq::variance1}, we can use that the second moment of $D_i \ {\buildrel d \over  =}\  \mathrm{Poi}(\rho)+1+\ind\{i\ith \text{ explored is blue}\}$ can be bounded by some constant $M_2$ independent of $i$. Hence, again cutting the sum at $K$, the sum of the first $K$ terms is a.s.\ finite. For the rest, we can use Lemma \ref{lemma:si} again, and then Markov's inequality yields:
\be\label{eq::squares} \ba \P\!\left(\sum_{i=K+1}^{C\sqrt{n}}\!\! \frac{D_i^2}{\left(i^2\lambda^2(1+o(i^{-1/2+\ve})^2\right)} \geq \!\!\sum_{i=K+1}^{C\sqrt{n}}\!\! \frac{M_2}{\left(i^2 2 \lambda^2 \right)} \cdot \log\log n \right) \!\leq \frac{1}{\log\log n}. \ea \ee
Note that $\sum_{i=1}^{\infty} M^2/(i^2 2 \la^2) \leq \tfrac{M_2\pi^2}{12} $. Combining the two estimates for the two terms in \eqref{eq::variance1}, we see that the variance tends to infinity w.h.p. As a result, the term $B_1$ in \eqref{eq::b123} satisfies a CLT.
\subsubsection{Term $B_2$}
Similarly as for the term $B_1$, we cut the sum at $K$ given by Lemma \ref{lemma:si} and write
\[ \ba \frac{\sum_{i=1}^{C\sqrt{n}}D_i/S_i(1-D_i/S_i)}{\frac{\lambda+1}{2\lambda} \log n} = \frac{\sum_{i=1}^{K}D_i/S_i(1-D_i/S_i)}{\frac{\lambda+1}{2\lambda} \log n} + \frac{\sum_{i=K+1}^{C\sqrt{n}}D_i/S_i}{\frac{\lambda+1}{2\lambda} \log n} - \frac{\sum_{i=1}^{C\sqrt{n}}D_i^2/S_i^2}{\frac{\lambda+1}{2\lambda} \log n}.
\ea \]
The first fraction tends to 0, as the numerator is a.s. finite. For the numerator in the third term, we can use \eqref{eq::squares} again, which shows that the third term tends to zero w.h.p.  We have yet to show that the second term tends to $1$.
Let $\cF_n = \sigma(D_1,...,D_n)$ be the filtration generated by the random variables $D_i$. Then
\be\label{eq::cond-mean} \frac{\sum_{i=K+1}^{C\sqrt{n}}D_i/S_i}{\frac{\lambda+1}{2\lambda} \log n}
= \frac{\sum_{i=K+1}^{C\sqrt{n}} \frac{D_i- \E[D_i|\cF_{i-1}]}{i\lambda(1+o(i^{-1/2+\ve}))} } {\frac{\lambda+1}{2\lambda} \log n} + \frac{\sum_{i=K+1}^{C\sqrt{n}} \frac{\E[D_i|\cF_{i-1}]}{i\lambda(1+o(i^{-1/2+\ve}))}} {\frac{\lambda+1}{2\lambda} \log n}. \ee
For the first term of the rhs of \eqref{eq::cond-mean}, we will use Chebyshev's inequality. 
For this, elementary calculation using tower rule yields that 
\[ \ba \Var\left[ \sum_{i=K+1}^{C\sqrt{n}} \frac{D_i-\E[D_i|\cF_{i-1}]}{i\lambda(1+o(i^{-1/2+\ve}))}  \right]
\le \sum_{i,K+1}^{C\sqrt{n}} \frac{\E [ D_i^2|\cF_{i-1}]}{\lambda^2 i^2 (1+o(i^{-1/2+\ve})^2}
\ea \]
Since $\E [ D_i^2|\cF_{i-1}]\le M_2$ as in Section \ref{ss::b1}, we get that the rhs is at most $M_2 \pi^2/(3\la^2)$.
Then Chebyshev's inequality yields
\be \label{vardicond} \ba &\P \left( \left| \sum_{i=K+1}^{C\sqrt{n}}  \frac{D_i- \E[D_i|\cF_{i-1}]}{i\lambda(1+o(i^{-1/2+\ve}))} \right| \geq \log\log n \cdot \frac{\pi^2}3 \frac{M_2}{\lambda^2} \right) \leq \frac1{(\log\log n)^2}.
\ea \ee
This implies that the first term in \eqref{eq::cond-mean} tends to $0$ w.h.p.

Now to show that the second term in \eqref{eq::cond-mean} tends to $1$, we use a corollary of Theorem \ref{theorem:actives} (see \cite{AthNey72BP}), stating that the vector $\left( \frac{S_{i-1}^R}{S_{i-1}}, \frac{S_{i-1}^B}{S_{i-1}}\right) \to (\pi_R,\pi_B)$ a.s. Further analysis (in particular, the central limit theorem about $(S_{i}^R, S_i^B)$ in \cite{Jan04functional}) yields that the error term is at most of order $i^{-1/2+\ve}$. Hence,  using that $D_i \ {\buildrel d \over  =}\  \mathrm{Poi}(\rho)+1+\ind\{i\ith \text{ explored is blue}\}$ and the definition of $\la$ it is elementary to show that
\[ \E[D_i|\cF_{i-1}] = (\lambda+1)(1+o(i^{-1/2+\ve})) \]
Substituting this into the sum, we have
\be\label{eq::term3} \ba \frac{\sum_{i=K+1}^{C\sqrt{n}} \frac{\E[D_i|\cF_{i-1}]}{i\lambda (1+ o(i^{-1/2+\ve}))}} {\frac{\lambda+1}{2\lambda} \log n}
= \frac{\sum_{i=K+1}^{C\sqrt{n}} 1/i }{\log n/2} + \frac{\sum_{i=K+1}^{C\sqrt{n}} o(i^{-1/2+\ve})/i} {\log n/2} \ea \ee
The first term on the rhs, introducing a constant error term $\delta$ from the integral approximation, equals
\be\label{eq::term45}  \frac{\sum_{i=K+1}^{C\sqrt{n}} 1/i }{ \log n/2} = \frac{\log(C\sqrt{n})-\log(K+1)+\delta}{\log n/2}
 \toninf 1, \ee
 since $C$ is a tight random variable.
The second term in \eqref{eq::term3} is at most $\sum_{i=0}^{\infty}O(i^{-3/2+\varepsilon})$ is summable and finite, hence divided by $\log n$ it tends to 0. Combining everything, we get that $B_2$ in \eqref{eq::b123} tends to $1$ w.h.p. 

\subsubsection{Term $B_3$}
As before, we cut this sum at $K$ given by Lemma \ref{lemma:si}, and the sum of the first $K$ terms divided by $\log n$ tends to $0$ since $D_i/S_i<1$.
When we consider the rest of the sum, we use the approximation of $S_i$ (given by Lemma \ref{lemma:si}) and add and subtract $\E[D_i|\cF_{i-1}]$ again:
\be\label{eq::term4} \frac{ \sum_{i=K+1}^{C\sqrt{n}} \frac{D_i}{S_i} - \frac{\lambda+1}{2\lambda}\log n}{\sqrt{\frac{\lambda+1}{2\lambda}\log n}}
= \frac{ \sum_{i=K+1}^{C\sqrt{n}} \frac{D_i - \E[D_i|\cF_{i-1}]}{i\lambda(1+o(i^{-1/2+\ve}))} } {\sqrt{\frac{\lambda+1}{2\lambda}\log n}}
+ \frac{ \left(\sum_{i=K+1}^{C\sqrt{n}} \frac{\E[D_i|\cF_{i-1}]}{i\lambda(1+o(i^{-1/2+\ve}))} \right) - \frac{\lambda+1}{2\lambda}\log n}{\sqrt{\frac{\lambda+1}{2\lambda}\log n}}
\ee
The numerator of the first term on the rhs has been treated in  \eqref{vardicond} and is w.h.p of order at most $\log\log n$, hence the first term on the rhs tends to $0$ w.h.p.
For the second term on the rhs of \eqref{eq::term4} we can use \eqref{eq::term3} and \eqref{eq::term45}, and then it is at most
\[ \ba  \frac{\lambda+1}{\lambda} \cdot \frac{ \log C - \log (K+1) +\delta+ \sum_{i=K+1}^{C\sqrt{n}} o(i^{-1/2+\ve})/i }{\sqrt{\frac{\lambda+1}{2\lambda}\log n}}
\to 0\ea \] almost surely, since $C$ is a tight random variable. This shows that the term $B_3$ in \eqref{eq::b123} tends to $0$ w.h.p, and finishes the proof of the CLT for the generation of the connecting vertex in $\SWT^V$, see \eqref{eq::gyv}.

\subsection{Generation of the connecting vertex in $\SWT^U$}
For the generation of $Y$ in $\SWT^U$, we have to use a different approach. This is because the label of the connecting vertex is chosen uniformly among the active vertex of $\SWT^V$ but is not necessarily uniform over the active vertices in $\SWT^U$. Indeed, it is a longish but elementary calculation to show that conditioned on the event that a connection happens, 
any active red label in $\SWT^U$ is chosen with asymptotic probability $(\rA^{(U)}(t_n))^{-1} (1-\tfrac{\pi_R}2)/(1-\tfrac{\pi_R^2}2)$, while any active blue label is chosen with asymptotical probability  $(\rA^{(U)}(t_n))^{-1}1/(1-\tfrac{\pi_R^2}2)$, where $\rA^{(U)}(t_n)$ is the total number of active vertices in $\SWT^U$. 
However, the following claim is still valid and will be enough to show the needed CLT:
\begin{claim}\label{claim::uni}
Conditioned on the connecting vertex having a label of a certain color in $\SWT^U$, with high probability, it is chosen uniformly at random among the active labels of that color in $\SWT^U$.
\end{claim}
\begin{proof} We show the statement first for color blue.
Recall that a blue label was chosen uniformly in $[n]$. Since the restriction of a uniform distribution to any set is again uniform, the probability of connection is the same for any particular blue active label among the \emph{different} labels in the active blue labels in $\SWT^U$. Recall that the number of different labels (that are neither thinned nor ghost) is called the effective size and it is treated in Corollary \ref{cor:effectivesize}.

The problem here is though, that some labels in the branching process approximation are multiply active, and these are neither thinned nor called ghosts, and if chosen, they modify the uniform probability for the connection.\footnote{Consider a label $V_i$ in the BP of $\SWT^U$ that is active $m_i$ times, i.e., there are $m_i$ individuals in the BP having label $V_i$. The label $V_i$ in $\SWT^V$ is still chosen only with the same probability, (that is, $1/n$). Since which one of the $m_i$ individuals  has the minimal remaining lifetime is  uniformly distributed, every individual with label $V_i$ has probability $(m_i \rA^{(U)}_{B,e} (t_n))^{-1}$ to be the connecting vertex, conditioned on connection at a blue label (with $\rA^{(U)}_{B,e}$ being the effective size of blue labels).} 
However, Corollary \ref{cor:effectivesize} implies that the fraction of multiple actives tends to $0$  at time $t_n$. Hence if we pick an active label in $\rA_B^U(t_n)$ it has multiplicity $1$ w.h.p., including red and blue instances. (This also implies that asymptotically, the label has a well defined color.) Hence with high probability, at the connecting vertex, we have a uniform distribution over all possible blue active labels.\\
An analogous argument can be carried through for red active labels as well, using the fact that the centre of the interval where they belong to is chosen uniformly, and the fact that multiple red labels have proportion tending to $0$ at time $t_n$.
\end{proof}
To finish the central limit theorem of $G_U(Y)$, we use a general result of Kharlamov \cite{kharlamov69} about the generation of a uniformly chosen active individual in a given type-set in a multi-type branching process. For this, consider a type set $\mathcal S$ of a multi-type branching process, and let $\cA_{\mathcal S} = \cup_{q \in \mathcal S} \cA_q$ the set of active individuals with any type from the type-set $S$.
Then, \cite[Theorem 2]{kharlamov69} states that the generation of a uniformly chosen individual in $\cA_{\mathcal S}$ satisfies a central limit theorem with asymptotic mean and variance that is \emph{independent} of the choice of $\mathcal S$.\footnote{In fact, Kharlamov also specifies the asymptotic mean and variance, but that is for us given by another method, so for us this weak implication of \cite[Theorem 2]{kharlamov69} is sufficient.}

To apply this result, first pick $\mathcal S:=\{R,B\}$ in our case. Then, the statement simply turns into a CLT of the generation of a uniformly picked active individual. 
We have seen when treating $G^{(V)}(Y)$ that the asymptotic mean and  variance are both $\frac{\la+1}{2\la} \log n$ in this case.

Now apply the result again for  $\mathcal S:= \{R\}$ and $\mathcal S:= \{B\}$, separately. Combined with the previous observation, we get that an individual chosen uniformly at random with color blue/red, respectively, also satisfies a CLT with the same asymptotic mean and variance. 
%
This, combined with Claim \ref{claim::uni}, implies that whether $Y$ is red or blue in $\SWT^U$, its generation $G^{(U)}(Y)$ admits a central limit theorem with mean and variance $\frac{\la+1}{2\la} \log n$. This completes the proof of Theorem \ref{thm:clt}.


\bibliographystyle{abbrv}
\bibliography{vikibib}

\begin{thebibliography}{10}

\bibitem{NWmixingtime}
L.~Addario-Berry and T.~Lei.
\newblock The mixing time of the {N}ewman-{W}atts small world.
\newblock In {\em Proceedings of the Twenty-third Annual ACM-SIAM Symposium on
  Discrete Algorithms}, SODA '12, pages 1661--1668. SIAM, 2012.

\bibitem{epiweb1}
{AFMC Primer on Population Health}.
\newblock Patterns of disease development in a population: the epidemic curve.
\newblock
  \url{http://phprimer.afmc.ca/Part2-MethodsStudyingHealth/Chapter7ApplicationsOfResearchMethodsInSurveillanceAndProgrammeEvaluation/Patternsofdiseasedevelopmentinapopulationtheepidemiccurve}.
\newblock Accessed 26 Feb 2015.

\bibitem{FPPcompetition11}
T.~{Antunovi{\'c}}, Y.~{Dekel}, E.~{Mossel}, and Y.~{Peres}.
\newblock {Competing first passage percolation on random regular graphs}.
\newblock {\em ArXiv e-prints}, Sept. 2011.

\bibitem{Asm77LIL}
S.~Asmussen.
\newblock Almost sure behavior of linear functionals of supercritical branching
  processes.
\newblock {\em Transactions of the American Mathematical Society},
  231(1):233--248, 1977.

\bibitem{AthNey72BP}
K.~Athreya and P.~Ney.
\newblock {\em Branching Processes}.
\newblock Die Grundlehren der mathematischen Wissenschaften in
  Einzeldarstellungen. Springer Berlin Heidelberg, 1972.

\bibitem{BarReicont}
A.~D. Barbour and G.~Reinert.
\newblock Small worlds.
\newblock {\em Random Structures \& Algorithms}, 19(1):54--74, 2001.

\bibitem{BarReisdisc}
A.~D. Barbour and G.~Reinert.
\newblock Discrete small world networks.
\newblock {\em Electron. J. Probab.}, 11:no. 47, 1234--1283, 2006.

\bibitem{BarRei13epidemiccurve}
A.~D. Barbour and G.~Reinert.
\newblock Approximating the epidemic curve.
\newblock {\em Electron. J. Probab}, 18(54):1--30, 2013.

\bibitem{BarHofKom14}
E.~Baroni, R.~v.~d. Hofstad, and J.~Komj{\'a}thy.
\newblock Fixed speed competition on the configuration model with infinite
  variance degrees: unequal speeds.
\newblock arXiv:1408.0475 math.PR, 2014.

\bibitem{BhaHogHoo10finitemean}
S.~Bhamidi, R.~v.~d. Hofstad, and G.~Hooghiemstra.
\newblock First passage percolation on random graphs with finite mean degrees.
\newblock {\em The Annals of Applied Probability}, 20(5):1907--1965, 2010.

\bibitem{BhaHofHoo11FPPer}
S.~Bhamidi, R.~van~der Hofstad, and G.~Hooghiemstra.
\newblock First passage percolation on the {E}rd{\H o}s–{R}{\'e}nyi random
  graph.
\newblock {\em Combinatorics, Probability and Computing}, 20:683--707, 9 2011.

\bibitem{BhaHofHoo12universality}
S.~{Bhamidi}, R.~{van der Hofstad}, and G.~{Hooghiemstra}.
\newblock {Universality for first passage percolation on sparse random graphs}.
\newblock {\em ArXiv e-prints}, Oct. 2012.

\bibitem{BhaHofKom14epidemicFPP}
S.~Bhamidi, R.~van~der Hofstad, and J.~Komj{\'a}thy.
\newblock The front of the epidemic spread and first passage percolation.
\newblock {\em J. Appl. Probab.}, 51A:101--121, 12 2014.

\bibitem{BolJanRio07inhom}
B.~Bollob{\'a}s, S.~Janson, and O.~Riordan.
\newblock The phase transition in inhomogeneous random graphs.
\newblock {\em Random Structures \& Algorithms}, 31(1):3--122, 2007.

\bibitem{Buhler71}
W.~B\"uhler.
\newblock Generations and degree of relationship in supercritical markov
  branching processes.
\newblock {\em Probability Theory and Related Fields}, 18(2):141--152, 1971.

\bibitem{Buhler72gen}
W.~J. B{\"u}hler et~al.
\newblock The distribution of generations and other aspects of the family
  structure of branching processes.
\newblock In {\em Proceedings of the Sixth Berkeley Symposium on Mathematical
  Statistics and Probability, Volume 3: Probability Theory}. The Regents of the
  University of California, 1972.

\bibitem{epiweb2}
{Centers for Disease Control and Prevention}.
\newblock Reports of salmonella outbreak investigations from 2013.
\newblock \url{http://www.cdc.gov/salmonella/outbreaks-2013.html}.
\newblock Accessed 26 Feb 2015.

\bibitem{DeiHof13competition}
M.~{Deijfen} and R.~{van der Hofstad}.
\newblock {The winner takes it all}.
\newblock {\em ArXiv e-prints}, June 2013.

\bibitem{Durrett07}
R.~Durrett.
\newblock {\em Random graph dynamics}, volume~20.
\newblock Cambridge university press, 2007.

\bibitem{ErdosRenyi60}
P.~Erd{\H o}s and A.~R{\'e}nyi.
\newblock On the evolution of random graphs.
\newblock {\em Publ. Math. Inst. Hung. Acad. Sci}, 5:17--61, 1960.

\bibitem{HamWel65}
J.~M. Hammersley and D.~J.~A. Welsh.
\newblock First-passage percolation, subadditive processes, stochastic
  networks, and generalized renewal theory.
\newblock In {\em Proc. {I}nternat. {R}es. {S}emin., {S}tatist. {L}ab., {U}niv.
  {C}alifornia, {B}erkeley, {C}alif}, pages 61--110. Springer-Verlag, New York,
  1965.

\bibitem{Hofs16}
R.~v.~d. Hofstad.
\newblock Random graphs and complex networks.
\newblock lecture notes.

\bibitem{HofHooZna05}
R.~v.~d. Hofstad, G.~Hooghiemstra, and D.~Znamenski.
\newblock Distances in random graphs with finite mean and infinite variance
  degrees.
\newblock {\em Electron. J. Probab.}, 12:no. 25, 703--766, 2007.

\bibitem{Janson123}
S.~Janson.
\newblock One, two and three times log n/n for paths in a complete graph with
  random weights.
\newblock {\em Combinatorics, Probability and Computing}, 8:347--361, 7 1999.

\bibitem{Jan04functional}
S.~Janson.
\newblock Functional limit theorems for multitype branching processes and
  generalized {P}\'olya urns.
\newblock {\em Stochastic Process. Appl.}, 110(2):177--245, 2004.

\bibitem{kharlamov69}
B.~Kharlamov.
\newblock The numbers of generations in a branching process with an arbitrary
  set of particle types.
\newblock {\em Theory of Probability \& Its Applications}, 14(3):432--449,
  1969.

\bibitem{KolKom12}
I.~{Kolossv{\'a}ry} and J.~{Komj{\'a}thy}.
\newblock First passage percolation on inhomogeneous random graphs.
\newblock {\em Advances of Applied Probability}, 47, 2015.

\bibitem{NewMooCri00}
C.~Moore and M.~E.~J. Newman.
\newblock Epidemics and percolation in small-world networks.
\newblock {\em Phys. Rev. E}, 61:5678--5682, May 2000.

\bibitem{NewWat99}
M.~Newman and D.~Watts.
\newblock Renormalization group analysis of the small-world network model.
\newblock {\em Physics Letters A}, 263(4–6):341 -- 346, 1999.

\bibitem{NewMooWat00}
M.~E.~J. Newman, C.~Moore, and D.~J. Watts.
\newblock Mean-field solution of the small-world network model.
\newblock {\em Phys. Rev. Lett.}, 84:3201--3204, Apr 2000.

\bibitem{NewWat99scaleFPP}
M.~E.~J. Newman and D.~J. Watts.
\newblock Scaling and percolation in the small-world network model.
\newblock {\em Phys. Rev. E}, 60:7332--7342, Dec 1999.

\bibitem{Res86Poi}
S.~I. Resnick.
\newblock Point processes, regular variation and weak convergence.
\newblock {\em Advances in Applied Probability}, 18(1):66--138, 1986.

\bibitem{Tor04ec}
M.~Torok, A.~Nelson, L.~Alexander, G.~C. Mejia, and P.~D. MacDonald.
\newblock Epidemic curves ahead.
\newblock {\em Focus on Field Epidemiology}, 1, 2004.

\bibitem{WatSto98}
D.~J. {Watts} and S.~H. {Strogatz}.
\newblock {Collective dynamics of `small-world' networks}.
\newblock {\em Nature}, 393:440--442, June 1998.

\end{thebibliography}

\end{document}